\documentclass[11pt]{article}

  \usepackage{latexsym,etex}
  \usepackage{comment,epigraph}
  \usepackage{amsmath,amsthm,amsfonts,amssymb,array,graphicx,epsfig,latexsym,color,amsbsy,
  tikz-cd, xspace}
  \usepackage{pictex,dcpic}
  \usepackage{epsf,enumerate}
  \usepackage[title]{appendix}
  \mathchardef\hyphenmathcode=\mathcode`\- %% to get minus sign to work in the code 
  \usepackage{breqn}
  \usepackage[margin=2.5cm]{geometry}
  \usepackage{listings}
  \usepackage[all,cmtip]{xy}
  \usepackage{mathtools}
  \usepackage[colorlinks,citecolor=blue,linkcolor=red!80!black]{hyperref}
  \usepackage[nameinlink]{cleveref}
  \usepackage[normalem]{ulem}

  \definecolor{codegreen}{rgb}{0,0.6,0}
  \definecolor{codegray}{rgb}{0.5,0.5,0.5}
  \definecolor{codepurple}{rgb}{0.58,0,0.82}
  \definecolor{backcolour}{rgb}{0.95,0.95,0.92}

  \lstdefinestyle{mystyle}{
      backgroundcolor=\color{backcolour},   
      commentstyle=\color{codegreen},
      keywordstyle=\color{magenta},
      numberstyle=\tiny\color{codegray},
      stringstyle=\color{codepurple},
      basicstyle=\ttfamily\footnotesize,
      breakatwhitespace=false,         
      breaklines=true,                 
      captionpos=b,                    
      keepspaces=true,                 
      numbers=left,                    
      numbersep=5pt,                  
      showspaces=false,                
      showstringspaces=false,
      showtabs=false,                  
      tabsize=2
  }

  \let\origlstlisting=\lstlisting
  \let\endoriglstlisting=\endlstlisting
  \renewenvironment{lstlisting}
      {\mathcode`\-=\hyphenmathcode
       \everymath{}\mathsurround=0pt\origlstlisting}
      {\endoriglstlisting}

  \newtheorem{thm}{Theorem}[section]
  \newtheorem{lemma}[thm]{Lemma}
  \newtheorem{cor}[thm]{Corollary}
  \newtheorem{corollary}[thm]{Corollary}
  \newtheorem{prop}[thm]{Proposition}
  
  \newtheorem*{main}{Main Theorem}{}

  \numberwithin{figure}{section}

  \theoremstyle{remark}

  \newtheorem{definition}[thm]{Definition}
  \newtheorem{remark}[thm]{Remark}
  
  \newtheorem*{definition*}{Definition}
  \newtheorem*{remark*}{Remark}

  \newtheorem{observation}[thm]{Observation}

  \newcommand{\norm}[1]{\left\lVert#1\right\rVert}
  \newcommand{\bigzero}{\mbox{\normalfont\Large 0}}
  \newcommand{\rvline}{\hspace*{-\arraycolsep}\vline\hspace*{-\arraycolsep}}
  \def\R{{\mathbb R}}

  \def\C{{\mathcal C}}
  \def\D{{\mathcal D}}
  \def\P{{\mathbb P}}
  \def\F{{\mathcal F}}

  \newcommand{\la}{\langle}
  \newcommand{\ra}{\rangle}
  \newcommand{\rk}{n} %rank of free group
  \newcommand{\free}[1][\rk]{\mathbb{F}_{#1}}
  \newcommand{\FreeF}[1][\rk]{\mathcal{FF}_{#1}}
  \DeclareMathOperator{\vol}{vol}
  \DeclareMathOperator{\Out}{Out}
  \DeclareMathOperator{\Aut}{Aut}
  \newcommand{\CV}[1][\rk]{\text{CV}_{#1}} %CV
  \newcommand{\cv}[1][\rk]{\text{cv}_{#1}} %CV
  \newcommand{\CVclo}[1][\rk]{\overline{\text{CV}}_{#1}} %closure of CV
  \newcommand{\cvclo}[1][\rk]{\overline{\text{cv}}_{#1}} %closure of CV
  \newcommand{\AT}{\mathcal{AT}}
  \newcommand{\Curr}[1][\rk]{\text{Curr}_{#1}} %space of currents
  \newcommand{\PCurr}[1][\rk]{\mathbb{P}\text{Curr}_{#1}} %space of projective currents
  \newcommand{\supp}{\text{supp}} %support of a current 
  
  \DeclareMathOperator{\id}{id}
  \DeclareMathOperator{\rank}{rank}
  \DeclareMathOperator{\diam}{diam}
  
  \newcommand{\Teich}{\ensuremath{\text{Teichm\"uller}}\xspace}

  \newcommand{\from}{\colon\,}

\begin{document}

%section{Title and abstract}

  \title{Limit sets of unfolding paths in Outer space}

  \author{Mladen Bestvina, Radhika Gupta, and Jing Tao}
\date{}
  \maketitle

  \abstract{We construct an unfolding path in Outer space which does not
  converge in the boundary, and instead it accumulates on the entire
  1-simplex of projectivized length measures on a non-geometric  arational $\R$-tree $T$.
  We also show that $T$ admits exactly two dual ergodic projective
  currents.} 
  
\section{Introduction}

  For the once-punctured torus, the Thurston compactification of the \Teich
  space by projective measured laminations coincides with the visual
  compactification of the hyperbolic plane. In this case, every geodesic
  ray has a unique limit point, and the dynamical behavior of the ray in
  moduli space is governed by the continued fraction of its limit point.
  For hyperbolic surfaces of higher complexity, \Teich space with the
  \Teich metric is no longer negatively curved \cite{Mas75, MW95} (or even
  Riemannian), and the Thurston boundary is no longer its visual boundary
  \cite{Ker80}. More surprisingly, geodesic rays do not always converge
  \cite{Len08, LLR}.

  For hyperbolic surfaces of higher complexity, another interesting
  phenomenon is the existence of nontrivial simplices in the Thurston
  boundary which correspond to measures on non-uniquely ergodic laminations.
  Particularly interesting is the case when the underlying lamination is
  minimal and filling, also called \emph{arational}. Constructions of
  non-uniquely ergodic arational laminations have a long history, and
  typically used flat structures on surfaces \cite{Veech, Sataev, KeynesNewton, Keane}. A topological construction was
  introduced in \cite{Gabai}. In \cite{LLR}, Leininger, Lenzhen and Rafi
  combined this topological approach with some arithmetic parameters akin to
  continued fractions. This allowed them to show that it is possible for
  the full simplex of measures on a non-uniquely ergodic arational
  lamination to be realized as the limit set of a \Teich geodesic ray.

  In this paper, we take the above construction into Culler-Vogtmann's
  Outer space \cite{CV:OuterSpace}. A Thurston-type boundary for Outer
  space is given by the set of projective classes of minimal, very small
  $\free$-trees \cite{CM:LengthFunction, BF:OuterLimits, CL:BoundaryOuterSpace, Horbez} and the action of
  $\Out(\free)$ extends continuously to the compactified space. The
  analogue of arational laminations are \emph{arational trees}; for
  example, trees dual to arational laminations on a once-punctured surface
  fall into this category. There are other examples, such as trees dual to
  minimal laminations on finite 2-complexes that are not surfaces, called
  {\it Levitt type}; and yet others, called \emph{non-geometric}, that do
  not come from the latter two constructions. The non-uniquely ergodic
  phenomenon for laminations has two natural analogues for $\free$-trees:
  one in terms of length measures on trees, giving rise to
  \emph{non-uniquely ergometric} trees \cite{Guirardel:Dynamics}, and the
  other in terms of currents, giving \emph{non-uniquely ergodic} trees, see
  \cite{CHL:NUE}. It is an open problem to determine whether these two
  notions coincide. An example of a Levitt type non-uniquely ergometric
  arational tree, modeled on Keane's construction, was given in
  \cite{reiner}. 
 
  In Outer space, the analogue of \Teich metric is the Lipschitz metric and
  that of \Teich geodesics are folding paths. However, unlike \Teich
  geodesics, a folding path in Outer space has a forward direction,
  reflecting the asymmetry of the Lipschitz metric. Even though the
  boundary of Outer space is not a visual boundary, a folding path always
  converges along its forward direction. Our main result is that this nice
  behavior does not persist in the backward direction; in fact, in the
  backward direction, folding paths can behave as badly as \Teich
  geodesics. Define an \emph{unfolding path} in Outer space to be a folding
  path with the backward direction. Our main result, as follows, is a
  direct analogue of the results of \cite{LLR}.   

  \begin{thm}
    There exists an unfolding path in Outer space of free group of rank 7
    which does not converge to a point in the boundary of Outer space. In
    fact, the limit set is a 1-simplex consisting of the full set of length
    measures on a non-geometric and arational tree $T$. Moreover, the set
    of projective currents dual to $T$ is also a 1-dimensional simplex. In
    particular, $T$ is neither uniquely ergometric nor uniquely ergodic.
  \end{thm}

  We use the framework of folding and unfolding sequences. Every such
  sequence tracks the combinatorics of an appropriate folding path, resp.\
  unfolding path, in Outer space.  An infinite folding sequence has a
  naturally associated limiting tree in the boundary of Outer space and an
  unfolding sequence has a naturally associated algebraic lamination,
  called the \emph{legal lamination}. The graphs in the folding sequence
  can be given compatible metrics, which are then  used to parametrize the
  different length measures supported on the limiting tree. Compatible edge
  thicknesses on the graphs of the unfolding sequence parametrize the
  different currents with support contained in the legal lamination. The
  latter can then be used to study the currents dual to the trees in the
  limit set of the unfolding sequence. See \cite{NPR} or our
  \Cref{sec:sequences} for definitions and more precise statements.

  Modeling the construction of \cite{LLR} on a 5-holed sphere, the folding
  and unfolding sequences we consider come from explicit sequences of
  automorphisms of the free group of rank 7. More explicitly, fix a
  non-geometric fully irreducible automorphism on three letters and extend
  it to an automorphism $\phi$ of $\free[7]$ by identity on the other four
  basis elements. Also let $\rho$ be a finite order automorphism of
  $\free[7]$ that rotates the support of $\phi$ off itself. For an integer
  $r$, set $\phi_r = \rho \phi^r$.  Given a sequence $(r_i)_{i \ge 1}$ of
  positive integers, define a sequence of automorphisms by \[\Phi_i =
  \phi_{r_1} \circ \cdots \circ \phi_{r_i}.\] From $(\Phi_i)_i$ we get an
  unfolding sequence using the train track map induced by $\phi_{r_i}$, and
  from $(\Phi_i^{-1})_i$ we get a companion folding sequence. The
  parameters $(r_i)_i$ play the role of the continued fraction expansion
  for the limiting tree of the folding sequence, and adjusting them
  produces different types of trees and behaviors of the unfolding
  sequence. In particular, we show that if the sequence $(r_i)_i$ satisfies
  certain arithmetic conditions and grows sufficiently fast, then the
  limiting tree is arational, non-geometric, non-uniquely ergodic, and
  non-uniquely ergometric. Moreover, the limit set of the unfolding
  sequence is the full simplex of length measures on the tree. We refer to
  \Cref{thm:conclusion} for the full technical statement.   

  To see how the parameters $(r_i)_i$ come into play, it is informative to
  look at the sequence of free factors $A_i = \Phi_i(A)$, where $A$ is the
  support of $\phi$. The $A_i$'s are the projection of the folding sequence
  to the free factor complex $\FreeF[7]$. By our construction, $A_i$ and
  $A_{i+1}$ are disjoint (meaning $\free[7]=A_i*A_{i+1}*B_i$ for some
  $B_i$), but $A_i, A_{i+2}$ are not, and $r_i$ measures the distance
  between the projections of $A_{i-2}$ and $A_{i+2}$ to the free factor
  complex of $A_i$. Morally, if $r_i$'s are sufficiently large, then
  $(A_i)_i$ forms a quasi-geodesic in $\FreeF[7]$. Hence, by
  \cite{BR:FFBound,H}, the limiting tree of the folding sequence is
  arational. In addition, we show that the tree is non-geometric.
  To get two currents on the tree, we take loops in the $A_i$'s, which
  correspond to currents on $\free$ and take projective limits of the odd
  and even subsequences. Non-unique ergometricity of the tree follows a
  similar principle. 
  
  Although our construction is general in spirit, the case of rank 7 is
  already fairly involved, and some computations used computer assistance.
  One issue is that there is no known algorithm to tell if a collection of
  free factors has a common complement. This issue appears in the proof of
  arationality of the limiting tree that led to the peculiar looking
  arithmetic conditions on the parameters; see \Cref{sec:arational tree}.
\newpage
  \subsubsection*{Outline}

  \begin{itemize}
    \item In \Cref{sec:background}, we review some background
      material, including train track maps, Outer space, currents, length
      measures, and arational trees. 
    \item In \Cref{sec:sequences} we discuss folding and unfolding
      sequences. We relate length measures on a folding sequence with the
      length measures on the limiting tree when it is arational. We also
      define the legal lamination for an unfolding sequence and state a
      result from \cite{NPR} relating the currents supported on the legal
      lamination with those of the unfolding sequence. 
    \item In \Cref{sec:construction}, we discuss our main
      construction to generate from a sequence $(r_i)_i$ of positive
      integers a sequence of automorphisms of $\free[7]$. The associated
      transition matrices for these automorphisms have block shapes which
      we use to analyze their asymptotic behavior. From each sequence of
      automorphisms and their inverses, we get a folding and unfolding
      sequence of graphs of rank 7 induced by their train track maps.
    \item In \Cref{sec:arational tree}, we show that under the right
      conditions on $(r_i)_i$, the folding sequence converges to a
      non-geometric and arational tree $T$ in boundary of Outer space of
      rank 7. To show arationality, we project the folding sequence to the
      free factor complex and show it is a quasi-geodesic.
    \item In \Cref{sec:unfolding}, we study the behavior of the
      unfolding sequence. The main result is that if the sequence $(r_i)_i$
      grows sufficiently fast, then the legal lamination of the unfolding
      sequence supports a 1-simplex of projective currents.
    \item In \Cref{sec:folding}, we show that if the sequence $(r_i)_i$
      grows sufficiently fast, then the limiting tree of the folding
      sequence supports a 1-simplex of projective length measures. In
      particular, the limiting tree is not uniquely ergometric. 
    \item In \Cref{sec:combination}, we relate the legal lamination
      of the unfolding sequence to the dual lamination of the limiting tree of
      the folding sequence. This shows the limiting tree is not uniquely
      ergodic.
    \item In \Cref{sec:limit unfolding}, we show that the
      unfolding sequence limits onto the full simplex of length measures on
      the limiting tree of the folding sequence, and thus does not have a
      unique limit in the boundary of Outer space. %$\partial \CV[7]$.
    \item In \Cref{sec:conclusion}, we collect the results to prove
      the main theorem.
    \item In \Cref{sec:appendix}, we prove a technical lemma
      about convergence of products of matrices.
  \end{itemize}

  \subsubsection*{Acknowledgments} 
  
  The authors gratefully acknowledge support: M.B.\ from NSF DMS-1905720,
  R.G.\ from the Sloan Foundation, and J.T.\ from NSF DMS-1651963.
  
\section{Background}
  
  \label{sec:background}
  
  Let $\free$ be the free group of rank $\rk$. We review some background on
  train track maps, Outer space, laminations, currents, arational trees and the free factor
  complex. 
  
  \subsection{Train track maps}

  We recall some basic definitions from \cite{BH:TrainTracks}. Identify
  $\free$ with $\pi_1(\mathrm{R}_n, \ast)$ where $\mathrm{R}_n$ is a rose
  with $\rk$ petals.  A \emph{marked graph} $G$ is a graph of rank $\rk$,
  all of whose vertices have valence at least three, equipped with a
  homotopy equivalence $m\from \mathrm{R}_n \to G$ called a \emph{marking}. 
  
  A \emph{length vector} on $G$ is a vector $\lambda \in \R^{|EG|}$ that
  assigns a positive number, i.e.\ a length, to every edge of $G$. The
  volume of $G$ with respect to $\lambda$ is the total length of all the
  edges of $G$. This induces a path metric on $G$ where the length of an
  edge $e$ is $\lambda(e)$. 
  
  A direction $d$ based at a vertex $v \in G$ is an oriented edge of $G$
  with initial vertex $v$. A \emph{turn} is an unordered pair of distinct
  directions based at the same vertex. A \emph{train track structure} on
  $G$ is an equivalence relation on the set of directions at each vertex $v
  \in G$. The classes of this relation are called \emph{gates}.  A turn
  $(d,d')$ is \emph{legal} if $d$ and $d'$ do not belong to the same gate,
  it is called \emph{illegal} otherwise.  A path is legal if it only
  crosses legal turns.
  
  A map $f \from  G \to G'$ between two graphs is called a \emph{morphism}
  if it is locally injective on open edges and sends vertices to vertices.
  If $G$ and $G'$ are metric graphs, then we can homotope $f$ relative to
  vertices such that it is linear on edges. Similarly, for an $\R$-tree
  $T$, a map $\tilde G\to T$ from the universal cover of $G$ is a morphism
  if it is injective on open edges. To a morphism $f\from  G \to G'$ we
  associate the {\it transition matrix} as follows: Enumerate the
  (unoriented) edges $e_1,e_2,\cdots,e_m$ of $G$ and $e_1', e_2', \cdots,
  e_n'$ of $G'$ . Then the transition matrix $M$ has size $n\times m$ and
  the $ij$-entry is the number of times $f(e_j)$ crosses $e_i'$, i.e.\ it
  is the cardinality of the set $f^{-1}(x)\cap e_j$ for a point $x$ in the
  interior of $e_i'$. If $f$ is in addition a homotopy equivalence, then
  $f$ is a \emph{change-of-marking}. 
 
  A homotopy equivalence $f\from  G \to G$ induces an outer automorphism of
  $\pi_1(G)$ and hence an element $\phi$ of $\Out(\free)$. If $f$ is a
  morphism then we say that $f$ is a \emph{topological representative} of
  $\phi$. A topological representative $f \from  G\to G$ induces a
  \emph{train track structure} on $G$ as follows: The map $f$ determines a
  map $Df$ on the directions in $G$ by defining $Df(e)$ to be the first
  (oriented) edge in the edge path $f(e)$.  We then declare $e_1\sim e_2$
  if $(Df)^k(e_1)=(Df)^k(e_2)$ for some $k\geq 1$.

  A topological representative $f\from  G \to G$ is called a \emph{train
  track map} if every vertex has at least two gates, and $f$ maps legal
  turns to legal turns and legal paths (equivalently, edges) to legal
  paths. Equivalently, every positive power $f^k$ is a topological
  representative. If $f$ is a train track map with transition matrix $M$,
  then the transition matrix of $f^k$ is $M^k$ for every $k\geq 1$. If $M$
  is \emph{primitive}, i.e.\ $M^k$ has positive entries for some $k \ge 1$,
  then Perron-Frobenius theory implies that there is an assignment of
  positive lengths to all the edges of $G$ so that $f$ uniformly expands
  lengths of legal paths by some factor $\lambda>1$, called the
  \emph{stretch factor} of $f$. 
  
  If $\sigma$ is a path (or a circuit) in $G$, we denote by $[\sigma]$ the
  reduced path homotopic to $\sigma$ (rel endpoints if $\sigma$ is a path).
  A path or circuit $\sigma$ in $G$ is called a \emph{periodic Nielsen
  path} if $[f^k(\sigma)]=\sigma$ for some $k\geq 1$.  If $k=1$, then
  $\sigma$ is a \emph{Nielsen path}.  A Nielsen path is \emph{indivisible},
  denoted INP, if it cannot be written as a concatenation of nontrivial
  Nielsen paths.

  The following lemma is an important property of train track maps. For a
  very rudimentary form, see \cite[Lemma 3.4]{BH:TrainTracks} showing that
  INPs have exactly one illegal turn, and for a more involved version see
  \cite{BFH:Laminations} (some details can also be found in
  \cite[Proposition 3.27, 3.28]{KL:InvariantLaminations}). We will need it
  for the proof of \Cref{lem:lose illegal turns} and include a proof here.

  \begin{lemma}\label{lem:lose illegal turns iwip}
    
    Let $h\from  G \to G$ be a train track map with a primitive transition
    matrix. There exists a constant $R >0$ such that for any edge path
    $\gamma$, either, 
    \begin{enumerate}
     \item the number of illegal turns in $[h^{R}(\gamma)]$ is less than
       that of $\gamma$, or, 
     \item $\gamma = u_1 v_1 u_2 v_2 \ldots u_n$ where each $u_i$
       is a legal subpath, possibly degenerate, and each
       $[h^{R}(v_i)]$ is a periodic INP.
    \end{enumerate}
  \end{lemma}

  \begin{proof}
    
    Let $\lambda > 1$ be the stretch factor of $h$, and equip $G$ with the
    metric so that $h$ uniformly expands the length of every legal path by
    $\lambda$. It goes back to the work of Thurston (see \cite{cooper})
    that there is a constant $BCC(h)$, called the bounded cancellation
    constant for $h$, such that if $\alpha\beta$ is a reduced edge path
    then $[h(\alpha)][h(\beta)]$ have cancellation bounded by $BCC(h)$. The
    existence of this constant is really a consequence of the Morse lemma
    and the fact that $h$ is a quasi-isometry. Define $C = BCC(h) /
    (\lambda-1)$.

    Here is the significance of $C$. To fix ideas let us assume that
    $\gamma$ has only one illegal turn, so $\gamma=\alpha\beta$ with both
    $\alpha,\beta$ legal. Say $\alpha$ has length $|\alpha|=C+\epsilon>C$.
    Then $h(\alpha)$ has length $\lambda |\alpha|$ and after cancellation
    with $h(\beta)$ the length is $\geq \lambda |\alpha|
    -BCC(h)=|\alpha|+\lambda\epsilon$. Thus assuming $[h^i(\gamma)]$ still
    has an illegal turn, the length of the initial subpath to the illegal
    turn has length growing exponentially in $i$, assuming it is long
    enough.

    We now prove the lemma for paths $\gamma=\alpha\beta$ with one illegal
    turn, and with $\alpha,\beta$ legal. Consider the finite collection of
    paths consisting of those with length at most $C$ with both endpoints
    at vertices, or with length exactly $C$ with only one endpoint at a
    vertex. Let $R$ be a number larger than the square of the size of this
    collection. If $[h^i(\gamma)]=\alpha_i\beta_i$ has one illegal turn
    (with $\alpha_i,\beta_i$ legal) for $i=1,2,\cdots,R$, then by the
    pigeon-hole principle there will be $i<j$ in this range so that the
    $C$-neighborhoods of the illegal turns of $[h^i(\gamma)]$ and
    $[h^j(\gamma)]$ are the same (if $\alpha_i$ or $\beta_i$ has length
    $<C$ this means $\alpha_i=\alpha_j$ or $\beta_i=\beta_j$). We can lift
    $h^{j-i}$ and $\gamma$ to the universal cover of the graph and arrange
    that (the lift of) $\gamma$ and $[h^{j-i}(\gamma)]$ have the same
    illegal turn.  Thus $h^{j-i}$ maps
    the terminal $C$-segment of $\alpha_i$ (or $\alpha_i$ itself) over
    itself (by the above calculation) and therefore fixes a point in
    $\alpha_i$, and similarly for $\beta_i$. The subpath of $[h^i(\gamma)]$
    between these fixed points is a periodic INP, proving the lemma in the
    case $\gamma$ has one illegal turn.

    The general case is similar. Write
    $\gamma=\gamma_1\gamma_2\cdots\gamma_s$ with all $\gamma_k$ legal and
    with the turn between $\gamma_k$ and $\gamma_{k+1}$ illegal. Also
    assume that $[h^i(\gamma)]$ has the same number of illegal turns for
    $i=1,\cdots,R$. We can write
    $[h^i(\gamma)]=\gamma_1^i\gamma_2^i\cdots\gamma_s^i$ with all
    $\gamma_k^i$ legal and the turns between them illegal. For each illegal
    turn corresponding to the pair $(k,k+1)$ there will be $i<j$ in this
    range so that the $C$-neighborhoods of the illegal turn in
    $[h^i(\gamma)]$ and in $[h^j(\gamma)]$ are the same. This gives fixed
    points of $h^{j-i}$ in $\gamma_k^i$ and $\gamma_{k+1}^i$ and these
    fixed points split $\gamma$ into periodic INPs and legal segments, as
    claimed. \qedhere 
  \end{proof}
  
  We will use the lemma in the situation that $h$ has no periodic INPs, in
  which case the conclusion is that whenever $\gamma$ is not legal, then
  $[h^{R}(\gamma)]$ has fewer illegal turns than $\gamma$.

  \subsection{Outer space and its boundary}\label{sec:outer space}
  
  An $\free$-tree is an $\mathbb{R}$-tree with an isometric action of
  $\free$. An $\free$-tree $T$ has dense orbits if some (every) orbit is
  dense in $T$. An $\free$-tree is called \emph{very small} if the action
  is minimal, arc stabilizers are either trivial or maximal cyclic, and
  tripod stabilizers are trivial. We review the definition of Outer space
  first introduced in \cite{CV:OuterSpace}. 
  
  \emph{Unprojectivized Outer space}, denoted by $\cv$, is the set of free,
  minimal, and simplicial $\free$-trees. By considering the quotient
  graphs, $\cv$ is also equivalently the set of marked metric graphs, i.e.\
  the set of triples $(G,m,\lambda)$, where $G$ is a graph of rank $n$ with
  all valences at least $3$, $m\from \mathrm{R}_n \to G$ is a marking, and
  $\lambda$ is a positive length vector on $G$. By
  \cite{CM:LengthFunction}, the map of $\cv \to \R^{\free}$ given by $T
  \mapsto (\norm{g}_T)_{g \in \free}$, where $\norm{g}_T$ is the
  translation length of $g$ in $T$, is an inclusion. This endows $\cv$ with
  a topology. The closure $\cvclo$ in $\R^{\free}$ is the space of very
  small $\free$-trees \cite{BF:OuterLimits,CL:BoundaryOuterSpace}. The
  boundary $\partial \cv = \cvclo-\cv$ consists of very small trees that
  are either not free or not simplicial.
  
  Culler Vogtmann's \emph{Outer space}, $\CV$, is the image of $\cv$ in the
  projective space $\P\R^{\free}$. Elements in $\CV$ can also be described
  as free, minimal, simplicial $\free$-trees with unit covolume.
  Topologically, $\CV$ is a complex made up of simplices with missing
  faces, where there is an open simplex for each marked graph $(G,m)$
  spanned by positive length vectors on $G$ of unit volume. The closure
  $\CVclo$ of $\CV$ in $\P\R^{\free}$ is compact and the boundary $\partial
  \CV = \CVclo-\CV$ is the projectivization of $\partial \cvclo$. 

  The spaces $\cv$ and $\CV$ and their closures are equipped with a natural
  (right) action by $\Out(\free)$. That is, for $\Phi \in \Out(\free)$ and $T
  \in \cvclo$ the translation length function of $T \Phi$ on $\free$
  is $\norm{g}_{T \Phi} = \norm{\phi(g)}_T$, where $\phi$ is any lift of
  $\Phi$ to $\Aut(\free)$.

  \subsection{Laminations, currents and non-uniquely ergodic trees}

  In \cite{BFH:Tits}, Bestvina, Feighn and Handel defined a dynamic
  invariant called the attracting lamination associated to a train track
  map. In this article, we will consider the more modern definition of a
  lamination as given in \cite{CHLI}. 

  Let $\partial \free$ denote the Gromov boundary of $\free$ and let
  $\Delta$ be the diagonal in $\partial \free \times \partial \free$. The
  \emph{double boundary} of $\free$ is $\partial^2 \free = (\partial \free
  \times \partial \free - \Delta) / \mathbb{Z}_2$, which parametrizes the
  space of unoriented bi-infinite geodesics in a Cayley graph of $\free$.
  By an \emph{(algebraic) lamination}, we mean a non-empty, closed and
  $\free$-invariant subset of $\partial^2 \free$. 

  Associated to $T\in \cvclo$ is a \emph{dual lamination} $L(T)$,
  defined as follows in \cite{CHLII}. For $\epsilon > 0$, let 
  \[ 
    L_{\epsilon}(T) = \overline{\{(g^{-\infty}, g^{\infty}) \quad | \quad
    \norm{g}_T < \epsilon, g \in \free \}},
  \]
  so $L_{\epsilon}(T)$ is a lamination and set $L(T) =
  \bigcap_{\epsilon>0}L_{\epsilon}(T)$. Elements of $L(T)$ are called
  \emph{leaves}. For trees in $\cv$, $L(T)$ is empty. 
  
  A \emph{current} is an additive, non-negative, $\free$-invariant function
  on the set of compact open sets in $\partial^2 \free$. Equivalently, it
  is an $\free$-invariant Radon measure on the $\sigma$-algebra of Borel
  sets of $\partial^2 \free$. Let $\Curr$ denote the space of currents,
  equipped with the weak* topology. The quotient space of $\PCurr$ of
  projectivized currents (i.e.\ homothety classes of nonzero currents) is
  compact.

  For $\mu \in \Curr$, let $\supp(\mu) \subset \partial^2  \free$ denote
  the support of $\mu$, which is in fact a lamination.  For $T \in \cvclo$
  and $\mu \in \Curr$, if $\supp(\mu) \subseteq L(T)$, then we say $\mu$ is
  \emph{dual} to $T$. Denote by $\C(T)$ the convex cone of currents dual to
  $T$ and by $\P\C(T)$ the set of projective currents dual to $T$.
  $\P\C(T)$ is a compact, convex space and its extremal points are called
  the \emph{ergodic currents} dual to $T$. We say $T$ is \emph{uniquely
  ergodic} if there is only one projective class of currents dual to $T$,
  and \emph{non-uniquely ergodic} otherwise. In \cite{CH:ErgodicCurrents},
  the authors show that if $T \in \partial \cv$ has dense orbits, then
  $\P\C(T)$ is the convex hull of at most $3\rk-5$ projective classes of
  ergodic currents dual to $T$.  

  In \cite{KL:IntersectionForm}, Kapovich and Lustig established a
  length pairing, $\la \cdot, \cdot \ra$, between $\cvclo$ and the space
  of measured currents $\Curr$. They also showed in \cite[Theorem
  1.1]{KL:ZeroLength} that for $T \in \cvclo$ and $\mu \in \Curr$, $\la
  T, \mu \ra = 0$ if and only if $\mu$ is dual to $T$.

  Given two trees $T$ and $T'$, we say a map $h\from  T \to T'$ is
  \emph{alignment-preserving} if whenever $b \in T$ is contained in an arc
  $[a,c] \subset T$, then $h(b)$ is contained in the arc $[h(a),h(c)]$.

  \begin{thm}[\cite{CHL:NUE}] \label{thm:CHL}
    Let $T, T' \in \partial \CV$ be two trees with dense orbits. The
    following are equivalent: 
    \begin{itemize}
     \item $L(T) = L(T')$.
     \item There exists an $\free$-equivariant alignment-preserving
       bijection between $T$ and $T'$. 
    \end{itemize}
  \end{thm}

  \subsection{Length measures and non-uniquely ergometric trees}

  Since $\R$-trees need not be locally compact, classical measure theory is
  not well-suited for them. In \cite{Paulin:Habilitation}, a \emph{length
  measure} was introduced for $\R$-trees. See \cite{Guirardel:Dynamics} for
  details. 
  
  A \emph{length measure} on an $\free$-tree $T$ is a collection of finite
  Borel measures $\lambda_I$ for every compact interval $I$ in $T$ such
  that if $J \subset I$, then $\lambda_J = (\lambda_I)|_J$. We require the
  length measure to be invariant under the $\free$ action. The collection
  of the Lebesgue measures of the intervals of $T$ is $\free$-invariant,
  and this will be called the Lebesgue measure of $T$. A length measure
  $\lambda$ is \emph{non-atomic} or \emph{positive} if every $\lambda_I$ is
  non-atomic or positive. If every orbit is dense in some segment of $T$,
  then $T$ cannot have an invariant measure with atoms. Further, if $T$ is
  \emph{indecomposable}, that is, if for any pair of nondegenerate arcs $I$
  and $J$ in $T$, there exist $g_1,\ldots,g_m \in \free$, such that $I
  \subset \bigcup g_iJ$ and $g_i J \cap g_{i+1} J$ is nondegenerate, then
  every nonzero length measure is positive (in fact, the condition of
  mixing \cite{Guirardel:Dynamics} suffices).  
  
  Let $\D(T)$ be the cone of $\free$-invariant length measures on $T$, with
  projectivization $\P\D(T)$, i.e.\ the homothety classes of
  $\free$-invariant length measures on $T$. $\P\D(T)$ is a compact convex
  set and we will call its extremal points the \emph{ergodic length
  measures} on $T$. When $T$ has dense orbits there are at most $3n-4$ such
  measures for any $T$, see \cite[Corollary 5.2, Lemma
  5.3]{Guirardel:Dynamics} and $\D(T)$ is naturally a subset of $\partial
  cv_n$. In fact, 
  
  \begin{lemma}\cite{Guirardel:Dynamics}\label{lem:D(T)}
    If $T \in \cv$ is indecomposable then $\mathcal{D}(T)$ is in one-to-one
    correspondence with the set of isometry classes of $\free$-invariant
    metrics on $T$, denoted $X_T \subset \cv$.  
  \end{lemma}
  \begin{proof}
    Let $\lambda \in \mathcal{D}(T)$ be a length measure on $T$. Consider
    the pseudo-metric $d_{\lambda}$ on $T$ where $d_{\lambda}(x,y) =
    \lambda([x,y])$ for $x, y \in T$. In fact, since $T$ is indecomposable,
    $d_{\lambda}$ is a metric on $T$. 
    For the converse, let $T' \in X_T$. Then the pull back of Lebesgue
    measure on $T'$ under identity map $\id \from T \to T'$ gives a
    positive length measure on $T$. \qedhere 
  \end{proof}
  
  We say $T$ is \emph{uniquely ergometric} if there is only one
  projective class of length measures on $T$, which necessarily is the
  homothety class of the Lebesgue measure on $T$. It is called
  \emph{non-uniquely ergometric} otherwise.
  
  \subsection{Arational trees and the free factor complex}
  \label{sec:arational trees}

  For a tree $T \in \cvclo$ and a free factor $H$ of $\free$, let $T_H$
  denote the minimal $H$-invariant subtree of $T$ (this tree is unique
  unless $H$ fixes an arc). A tree $T \in \partial \cv$ is \emph{arational}
  if every proper free factor $H$ of $\free$ has a free and simplicial
  action on $T_H$. 
  By \cite{R:ReducingSystems} every arational tree is free
  and indecomposable or it is the dual tree to an arational measured
  lamination on a surface with one puncture. The arational trees of the
  first kind are either Levitt type or non-geometric. 

  Let $\AT \subset \partial \CV$ denote the set of arational trees with the
  subspace topology. Using \Cref{lem:D(T)}, define an equivalence relation
  $\sim$ on $\AT$ by ``forgetting the metric'', that is, $T \sim T'$ if $T'
  \in \P \D(T)$, and endow $\AT/\sim$ with the quotient topology. The
  following lemma is implicit in \cite{Guirardel:Dynamics} and we include a
  proof for completeness. 
  
  \begin{lemma}\label{lem:arational equivalence}
    Let $T, T'$ be arational trees. Then $T \sim T'$ if and only if $L(T) =
    L(T')$.  
  \end{lemma}

  \begin{proof}

    If $T \sim T'$ then the identity map $\id \from T \to T'$ is an alignment preserving
    bijection. Therefore, by \Cref{thm:CHL}, $L(T) =
    L(T')$. 
   
    If $L(T) = L(T')$, then by \Cref{thm:CHL} there is an alignment
    preserving bijection $f \from T \to T'$. Pulling back the Lebesgue
    measure on $T'$ induces a length measure on $T$, and the corresponding
    metric $d_\mu$ on $T$ is isometric to $T'$, so $T' \sim T$.
    \qedhere

  \end{proof}

  The \emph{free factor complex} $\FreeF$ is a simplicial complex whose
  vertices are given by conjugacy classes of proper free factors of $\free$
  and a $k$-simplex is given by a nested chain $[A_0] \subset [A_1] \subset
  \cdots \subset [A_k]$. When the rank $n=2$ the definition is modified and
  an edge connects two conjugacy classes of rank 1 factors if they have
  complementary representatives. The free factor complex can be given a
  metric as follows: identify each simplex with a standard simplex and
  endow the resulting space with path metric. By result of
  \cite{BF:FreeFactorComplex}, the metric space $\FreeF$ is Gromov
  hyperbolic. The Gromov boundary of $\FreeF$ was identified with
  $\AT/\sim$ in \cite{BR:FFBound} and \cite{H}.

  There is a projection map $\pi\from \CV \to \FreeF$ defined as follows
  \cite[Section 3]{BF:FreeFactorComplex}: for $G \in \CV$, $\pi(G)$ is the
  collection of free factors given by the fundamental group of proper
  subgraphs of $G$ which are not forests. This map is coarsely well
  defined, that is, $\diam_{\FreeF}(\pi(G)) \le K$ for some universal $K$.
  Note that if $G, G'$ belong to the same open simplex of $\CV$, then
  $\pi(G)=\pi(G')$, so the projection of a simplex of $\CV$ has uniformly
  bounded diameter.

\section{Folding and unfolding sequences}

  \label{sec:sequences}

  In this section we introduce (un)folding sequences and review some
  work of Namazi-Pettet-Reynolds \cite{NPR}.
  
  A \emph{folding/unfolding sequence} is a sequence 
  \[
    \begin{tikzcd} 
      & G_a \arrow[r]  
      & \cdots \arrow{r}
      & G_{-1} \arrow{r}
      & G_0 \arrow{r}
      & G_1 \arrow{r}
      & \cdots \arrow[r] 
      & G_b
    \end{tikzcd}
  \] 
  of graphs, together with maps $f_i\from  G_i \to G_{i+1}$ such that for
  any $j \leq i$, $f_{i-1} \circ f_{i-2} \circ \cdots \circ f_j\from  G_j
  \to G_i$ is a change-of-marking morphism. Equivalently, a sequence as
  above is called a folding/unfolding sequence, if there exists a train
  track structure on each $G_i$ and $f_{i-1} \circ f_{i-2} \circ \cdots
  \circ f_j$ maps legal paths to legal paths. We allow the sequence to be
  infinite in one or both directions. We assume that a marking on $G_0$ has
  been specified, so a folding/unfolding sequence determines a sequence of
  open simplices in Outer space. 
  
  Let $Q_i$ be the transition matrix of $f_{i}$. A \emph{length measure}
  for a folding/unfolding sequence $(G_i)_{a \leq i \leq b}$ is a sequence
  $(\lambda_i)_{a \leq i \leq b}$, where $\lambda_i \in \R^{|EG_i|}$ is a
  length vector on $G_i$, and for $a \leq i < b$, 
  \[ \lambda_i  = Q_i^T\lambda_{i+1}.\]
  In this way, $f_i$ restricts to an local isometry on every edge of $G_i$.
  When $b < \infty$, a length vector on $G_b$ determines a length measure
  on the sequence. When the sequence is infinite in the forward direction
  we denote by $\D((G_i)_i)$ the space of length measures on $(G_i)_i$, and
  $\P\D((G_i)_i)$ its projectivization. Observe that the dimension of
  $\D((G_i)_i)$ is bounded by $\liminf_{i \to \infty}|EG_i|$. 
  
  A \emph{current} for a folding/unfolding sequence $(G_i)_{a \leq i \leq
  b}$ is a sequence $(\mu_i)_{a \leq i \leq b}$, where $\mu_i \in
  \R^{|EG_i|}$ is a length vector on $G_i$ (but thought of as a vector of
  thicknesses of edges), and for $a \leq i < b$, we require $$\mu_{i+1}  =
  Q_i \mu_i.$$  Likewise, when the sequence is infinite in the backward
  direction, we denote by $\C((G_i)_i)$ the space of currents on $(G_i)_i$,
  and $\P\C((G_i)_i)$ its projectivization. The dimension of $\C((G_i)_i)$
  is bounded by $\liminf_{i \to - \infty}|EG_i|$.

  \subsection{Isomorphism between length measures}
    
  In this section, we identify the space of length measures on a folding
  sequence with that of the limiting tree, when it is an arational tree. 
  
  Consider a folding sequence of marked graphs of rank $n$
  \[
    \begin{tikzcd} 
      G_0 \arrow{r}{f_1}
      & G_1 \arrow{r}
      & \cdots \cdots \arrow{r}{f_i}
      & G_i \arrow{r}{f_{i+1}}
      & \cdots \cdots.
    \end{tikzcd}
  \] 
  Let $\tilde G_i$ be the universal cover of $G_i$, and let $\tilde f_i$ be
  a lift of $f_i$. For any positive length measure $(\lambda_i)_i \in
  \D((G_i)_i)$, we can realize $(\tilde G_i, \tilde \lambda_i)_i$ as a
  sequence in $\cv$, which can be ``filled in'' by a folding path in $\cv$
  (see \cite{BF:FreeFactorComplex} for details on folding paths). In
  particular, $(\tilde G_i,\tilde \lambda_i)_i$ always converges to a point
  $T \in \partial \cv$. Furthermore, we have morphisms $h_i \from  \tilde
  G_i \to T$ such that $h_i = h_{i+1} \tilde f_{i+1}$. With respect to the
  length measure $\tilde \lambda_i$, $\tilde f_i$ and $h_i$ restrict to
  isometries on edges \cite[Lemma 7.6]{BR:FFBound}. 
  
  Let $(U_i)_i$ be the sequence of open simplices $\CV$ associated to the
  sequence $(G_i)_i$.  Recall the projection map $\pi \from  \CV \to
  \FreeF$ is coarsely well-defined on simplices of $\CV$. We will say the
  folding sequence $(G_i)_i$ \emph{converges to an arational tree} $T$ if
  $\pi(U_i)$ converges to $[T] \in \partial \FreeF$. 
  
  \begin{prop} \label{prop:IsoLength}
    
    Suppose a folding sequence $(G_i)_i$ converges to an arational tree
    $T$. Then there is a linear isomorphism between $\D((G_i)_i)$ and
    $\D(T)$.  

  \end{prop}

  \begin{proof}
    
    Fix a positive length measure $(\lambda_i)_i \in \D((G_i)_i)$ and let
    $T \in \partial \cv$ be the limiting tree of $(\tilde
    G_i,\tilde\lambda_i)$ with corresponding morphism $h_i \from  \tilde
    G_i \to T$. Recall from \Cref{sec:arational trees} that if $T$ is
    arational, then we can identify $\D(T)$ with the subspace of
    $\free$-metrics on $T$ in $\partial \cv$. We will let $\lambda \in
    \D(T)$ be a length measure, and $T_\lambda$ its image in $\partial
    \cv$.  
    
    By \cite{BR:FFBound}, if $\pi(U_i)$ converges to $[T''] \in \partial
    \FreeF$, then for any positive $(\lambda_i')_i \in \D((G_i)_i)$,
    $(\tilde{G}_i,\tilde \lambda_i')$ also converges to an arational tree
    $T' \in \partial \cv$, such that $[T''] = [T']=[T]$; in other words,
    $T' = T_{\lambda'}$ for some $\lambda' \in D(T)$. This gives a linear
    map $\D((G_i)_i) \to \D(T)$. 

    Conversely, for any positive length measure $\lambda' \in \D(T)$, we can
    use the morphism $h_i$ to pull back $\lambda'$ from $T$ to a length
    measure $\lambda_i'$ on $\tilde{G}_i$. The fact that $h_i = h_{i+1}
    \tilde f_{i+1}$ implies $(\lambda_i')_i \in \D((G_i)_i)$. Moreover, the
    sequence $(\tilde{G}_i,\tilde \lambda_i')_i$ converges to $T_{\lambda'}
    \in \cv$. This gives a linear map $\D(T) \to \D((G_i)_i)$ which is the
    inverse of $\D((G_i)_i) \to \D(T)$ defined above. This shows
    $\D((G_i)_i) \to \D(T)$ is an isomorphism. \qedhere
   \end{proof}

   \begin{remark}
     A more general statement of \Cref{prop:IsoLength} which doesn't
     involve the assumption that $T$ is arational can be found in
     \cite[Proposition 5.4]{NPR}, but we will not need such a general
     statement here.
   \end{remark}

  \subsection{Isomorphism between currents}

  In this section, we state an analogous result identifying the space of
  currents on an unfolding sequence with the space of currents of a legal
  lamination associated to a unfolding sequence. We record some definitions
  from \cite{NPR} first. 
  
  Consider an unfolding sequence of marked graphs of rank $n$
  \[
    \begin{tikzcd} 
      & \cdots \cdots
      \arrow{r}{f_{i+1}}
      & G_i \arrow{r}{f_i}
      & \cdots \cdots \arrow{r}{f_2}
      & G_1 \arrow{r}\arrow{r}{f_1}
      & G_0. 
    \end{tikzcd}
  \] 
  Denote the composition $F_i = f_1 \circ \cdots \circ f_i$. Let
  $\Omega^L_{\infty}(G_i)$ denote the set of bi-infinite legal paths in
  $G_i$. Define the \emph{legal lamination} of the unfolding sequence
  $(G_i)_i$ to be 
  \[\Lambda = \bigcap_{i} F_i(\Omega^L_{\infty}(G_i)) \subseteq
  \Omega_{\infty}^L(G_0).\]  Use the marking on $G_0$ to identify
  $\partial^2 \pi_1(G_0)$ with $\partial^2 \free$. The preimage, in
  $\partial^2 \free$, of the lift of $\Lambda$ to $\partial^2 \pi_1(G_0)$
  is a lamination $\tilde{\Lambda}$. We denote by $\C(\Lambda)$ the convex
  cone of currents supported on $\tilde{\Lambda}$, with projectivization
  $\P\C(\Lambda)$.

  An \emph{invariant sequence of subgraphs} is a sequence of nondegenerate
  (i.e.\ not forests) proper subgraphs $H_i \subset G_i$ such that $f_i$
  restricts to a morphism $H_i \to H_{i-1}$. We will need the following
  theorem from \cite{NPR}.
  
  \begin{thm} [Theorem 4.4 \cite{NPR}] \label{thm:NPR currents} 
    Given an unfolding sequence $(G_i)_{i \geq 0}$ without an invariant
    sequence of subgraphs and with legal lamination $\Lambda$, then there
    is a natural linear isomorphism between $\C((G_i)_i)$ and $\C(\Lambda)$.  
  \end{thm} 

\section{Main setup}
  
  \label{sec:construction}

  In this section, we will construct an unfolding sequence $(\tau_i)_i$ and
  a folding sequence $(\tau_i')_i$ in $\CV[7]$ that intersect the same
  infinite set of simplices, which we will eventually use to show the
  existence of a non-uniquely ergodic and ergometric tree. The construction
  is done via a family of outer automorphisms. We will describe these
  automorphisms and then analyze the asymptotic behavior of their train
  track maps.

  \subsection{The automorphisms}

  Let $\free[7]=\la a,b,c,d,e,f,g\ra$. Denote by $\bar{x}$ the inverse of
  $x \in \free[7]$. First consider the map induced on the 3-petaled rose by
  the automorphism \[\theta \from  a \mapsto b, b \mapsto c, c \mapsto ca
  \in \Aut(\mathbb{F}_3)\] and the map induced by the inverse automorphism
  \[\vartheta \from  a \mapsto \bar{b}c, b \mapsto a, c \mapsto b.\] Using
  $\theta$ and $\vartheta$ to also denote the corresponding graph
  maps, and using the convention that $a$ also denotes the initial
  direction of the oriented edge $a$, while $\bar{a}$ denotes the
  terminal direction, we
  have 
  \[
    D\theta^3 \from  
    \begin{tikzcd}
      a\arrow[r] & c \arrow[out=-120,in=-60,loop] & b\arrow[l] & \bar{a} 
      \arrow[loop right] & \bar{b} \arrow[loop right]  & \bar{c}
    \arrow[loop right] \end{tikzcd} 
  \]

  \[
    D\vartheta^3 \from  
    \begin{tikzcd}
      a\arrow[r] & \bar{c} \arrow[loop right] & b\arrow[r] & \bar{a} \arrow[loop right] & c
      \arrow[r]  & \bar{b} \arrow[loop right]
    \end{tikzcd} 
  \]

  \begin{observation}\label{obs} 
    From the structure of the above maps, for $n\equiv 0$ mod 3, $D\theta^n
    = D\theta^3$ and $D\vartheta^n = D\vartheta^3$.
  \end{observation}

  \begin{lemma}\label{lem:vartheta train track}
    The map on the 3-petaled rose labeled $a, b, c$ induced by $\vartheta$
    is a train track map with respect to the train track structure with
    gates $\{a, \bar{c}\}, \{b, \bar{a}\}, \{c,\bar{b}\}$. Moreover, this
    train track map does not have any periodic INPs.

    The map on the 3-petaled rose labeled $a, b, c$ induced by $\theta$
    is also a train track map with respect to the train track structure with
    gates $\{a, b, c\}, \{\bar{a}\}, \{\bar{b}\}, \{\bar{c}\}$ and it has
    one periodic Nielsen path (see \cite[Example 3.4]{BF:OuterLimits}). 
  \end{lemma}

  \begin{proof}
    The train track structure on the rose induces a metric on the graph
    coming from Perron-Frobenius theory. Every INP has length at most twice
    the volume of the graph, one illegal turn, and the endpoints are fixed.
    Since there are only finitely many fixed points in $G$, it is easy to
    enumerate all such paths and check if they are Nielsen. For periodic
    INPs one knows that the period is bounded by a function of the rank of
    $\free$ \cite{feighn-handel}, so one can take a suitable power and
    check for INPs (though there are more efficient ways, see \cite{ilya}).
    Coulbois' train track package \cite{coulbois} for the mathematics
    software system {\tt Sage} \cite{sage1} computes periodic INPs of train
    track maps. \qedhere   
  \end{proof}
  
  Now let $\phi\in \Aut(\free[7])$ be the automorphism: $$a\mapsto
  b,b\mapsto c,c\mapsto ca,d\mapsto d,e\mapsto e,f\mapsto f,g\mapsto g$$
  and $\rho\in \Aut(\mathbb{F}_7)$ be the rotation by 4 clicks: $$a\mapsto
  e, b\mapsto f, c\mapsto g, d\mapsto a, e\mapsto b, f\mapsto c, g\mapsto
  d.$$ Thus $\phi$ is the extension of $\theta$ by identity, and $\rho$
  rotates the support of $\phi$ off itself. 

  \begin{lemma}\label{lem:phi train track}
    For any $r \ge 3$, the map on the $7$-petaled rose induced by $\phi_r
    = \rho \phi^r$ is a train track map with respect to the train track
    structure with gates $$\{a,b,c\},\{d,e,f\}$$ and 8 more gates
    consisting of single half edges. The transition matrix $M_r$ has block
    form
    \[
    \begin{pmatrix}
      0&I\\
      B^r&0
    \end{pmatrix}
    \]
    where $I$ is the $4\times 4$ identity matrix, and $B$ is the transition
    matrix of $\theta$:
    \[
    B=
    \begin{pmatrix}
      0&0&1\\
      1&0&0\\
      0&1&1
    \end{pmatrix}
    \]
  \end{lemma}

  \begin{proof}
    By \Cref{obs}, we only have to check the lemma for $\phi_3,
    \phi_4, \phi_5$, which can be done by hand or using the train
    track package for {\tt Sage}. 
  \end{proof}

  \begin{lemma}\label{lem:psi train track}
    For any $r \ge 3$ and $r \equiv 0$ mod 3, the map on the $7$-petaled rose induced by
    $\psi_r = (\rho\phi^r)^{-1}$ is a train track map with respect to the train
    track structure with gates
    $$\{a,e,\bar{g}\},\{b, \bar{d}\},\{c, \bar{b}\},\{d,
    \bar{c}\},\{f,\bar{e}\},\{g,\bar{f}\},\{ \bar{a}\}$$
    The transition matrix $N_r$ has block form
    \[
    \begin{pmatrix}
      0&C^r\\
      I&0
    \end{pmatrix}
    \]
    where $I$ is the $4\times 4$ identity matrix, and $C$ is the transition
    matrix of $\vartheta$:
    \[
    C=
    \begin{pmatrix}
      0&1&0\\
      1&0&1\\
      1&0&0
    \end{pmatrix}
    \]
  \end{lemma}

  \begin{proof}
    By \Cref{obs}, we only have to check the lemma for $\psi_3$, which can
    be done by hand or using the train track package for {\tt Sage}.
  \end{proof}

  \subsection{Asymptotics of transition matrices}

  \label{sec:Matrices}
  Let $\theta$, $\vartheta$, $\phi_r$, $\psi_r$ be the maps defined in the
  last section. We now analyze the behavior of the transition matrices
  $M_r$ and $N_r$ for $\phi_r$ and $\psi_r$ respectively. 
  
  \begin{lemma}\label{lem:Y}
    Let $B$ be the transition matrix for $\theta$, with Perron-Frobenius
    eigenvalue $\lambda_B$. There exists a constant $\kappa_B>0$ such
    that if $r,s-r\to\infty$, then $$\frac 1{\kappa_B \lambda_B^s}
    M_r M_s\to Y$$ where $Y$ is an idempotent matrix of the form 
    \[Y = \begin{pmatrix} 
     u & p u & q u & 0 & 0 & 0 & 0 
    \end{pmatrix} 
    \text{ with }  
    u = 
    \begin{pmatrix} 
    0, u_1, u_2, u_3, 0, 0, 0 \end{pmatrix}^T 
    \text{ and } 
    p, q > 0,\]  
    and 
    $\begin{pmatrix} u_1,u_2,u_3\end{pmatrix}^T$ 
    is a Perron-Frobenius eigenvector of $B$.   
  \end{lemma}

  \begin{proof}  
    There exists a Perron-Frobenius eigenvector $x=(x_1,x_2,x_3)^T$ for $B$
    and constants $p, q > 0$ such that 
    \[ 
    P = \lim_{s \to \infty} \frac{B^s}{\lambda_B^s} =
    \begin{pmatrix} x & px & qx \end{pmatrix}.
    \] 
    We have
    \[
      M_rM_s
      =
      \begin{pmatrix}
        \begin{matrix}
          0 & 0 & 0
        \end{matrix}
        & \rvline & 
        \begin{matrix}
          0 & 0 & 0
        \end{matrix}  
        & \rvline & 1 \\
        \hline
        B^s
        & \rvline & \bigzero & \rvline & 
        \begin{matrix} 
         0 \\ 0 \\ 0 
        \end{matrix}
        \\
        \hline
        \bigzero & \rvline & B^r & \rvline &  
        \begin{matrix}
          0 \\ 0 \\ 0
        \end{matrix} 
      \end{pmatrix}
      \quad
      \Longrightarrow
      \quad
      \frac{1}{\lambda_B^s} M_rM_s
      \to 
      \begin{pmatrix}
        \begin{matrix}
          0 & 0 & 0
        \end{matrix}
        & \rvline & 
        \begin{matrix}
          0 & 0 & 0
        \end{matrix}  
        & \rvline & 0 \\
        \hline
        P 
        & \rvline & \bigzero & \rvline & 
        \begin{matrix} 
         0 \\ 0 \\ 0 
        \end{matrix}
        \\
        \hline
        \bigzero & \rvline & \bigzero & \rvline &  
        \begin{matrix}
          0 \\ 0 \\ 0
        \end{matrix} 
      \end{pmatrix} 
    \]
    The square of the limiting matrix above has a nonzero block where $P$
    is of the form $$(px_1+qx_2)P,$$ and zero elsewhere, so we set
    \[ \kappa_B=px_1+qx_2 \qquad \text{and} \qquad (u_1,u_2,u_3)^T =
    \frac{1}{\kappa_B} (x_1,x_2,x_3)^T. \qedhere \]
  \end{proof}

  We have a similar statement for the matrices $N_r$.

  \begin{lemma}\label{lem:Z}

    Let $C$ be the transition matrix for $\vartheta = \theta^{-1}$, with
    Perron-Frobenius eigenvalue $\lambda_C$. There exists a constant
    $\kappa_C>0$ such that if $r, s-r\to\infty$, then
    $$\frac 1{\kappa_C \lambda_C^s} N_s N_r \to Z,$$ where $Z$ is an idempotent
    matrix of the form 
    \[
      Z = 
      \begin{pmatrix} 
        0 & v & p v & q v & 0 & 0 & 0 
      \end{pmatrix} 
      \text{ with }  
      v = 
      \begin{pmatrix}  
        v_1 , v_2 , v_3 , 0 , 0 , 0 , 0
      \end{pmatrix}^T 
      \text{ and } 
      p, q > 0,
    \] 
    and $(v_1,v_2,v_3)^T$ is a Perron-Frobenius eigenvector of $C$.
  \end{lemma}

  \begin{proof}
    We observe that the matrix $N_s N_r$ has shape that is the transpose of
    the matrix in \Cref{lem:Y}, with powers of the PF matrix $C$
    forming the nonzero blocks: 
    \[
      N_sN_r=
    \begin{pmatrix}
      \begin{matrix}
        0\\
        0\\
        0
      \end{matrix}
      & \rvline & C^s & \rvline & \bigzero \\
      \hline
      \begin{matrix}
        0\\
        0\\
        0
      \end{matrix}
      & \rvline & \bigzero & \rvline & C^r \\
      \hline
      1 & \rvline &
      \begin{matrix}
        0 & 0 & 0
      \end{matrix}
      & \rvline &
      \begin{matrix}
        0 & 0 & 0
      \end{matrix}
    \end{pmatrix}
    \qedhere
    \]
  \end{proof}

  For future reference, we also record the following. 
    Let $P = \lim_{r
  \to \infty} B^r/\lambda_B^r$ and $Q = \lim_{r \to \infty}
  C^r/\lambda_C^r$. Set 
  \[
    M_\infty = 
    \lim_{r \to \infty} \frac{M_r}{\lambda^r_B} 
    =\begin{pmatrix}
    \begin{matrix}
      0
    \end{matrix}
    & \rvline & 0\\
    \hline
    P & \rvline & 0
    \end{pmatrix}
    \qquad \text{and} \qquad
    N_\infty = 
    \lim_{r \to \infty} \frac{N_r}{\lambda^r_C} 
    =\begin{pmatrix}
    \begin{matrix}
      0
    \end{matrix}
    & \rvline & Q\\
    \hline
    0 & \rvline & 0
    \end{pmatrix}
    \]

  \begin{lemma} \label{lem:MYMZ}
    
    There are $p,q,r,s > 0$ such that
    \[
    M_\infty Y = \begin{pmatrix} 
     y & py & qy & 0 & 0 & 0 & 0 
    \end{pmatrix} \text{ with }  
    y = 
    \begin{pmatrix} 
    0, 0, 0, 0, y_1, y_2, y_3 
    \end{pmatrix}^T,
    \]
    $(y_1,y_2,y_3)^T$ is a Perron-Frobenius eigenvector of $B$; and
    \[
      ZN_\infty = \begin{pmatrix} 
       0 & 0 & 0 & 0 & z & rz & sz 
      \end{pmatrix} 
      \text{ with }  
      z = 
      \begin{pmatrix}  
      z_1, z_2, z_3,  0 , 0 , 0 , 0 
      \end{pmatrix}^T, 
    \]
  and  $(z_1,z_2,z_3)^T$ is a Perron-Frobeninus eigenvector of $C$.
  \end{lemma}

  \subsection{Folding and unfolding sequence}
  
  Consider a sequence of positive integers $(r_i)_{i \ge 1}$ and the
  sequence of automorphisms $\phi_{r_i}$, with transition matrix $M_{r_i}$
  and $\phi_{r_i}^{-1} = \psi_{r_i}$ with transition matrix $N_{r_i}$. Let
  $\tau_i \to \tau_{i-1}$ (resp. $\tau_{i-1}' \to \tau_i'$) be the train
  track map induced on the rose by $\phi_{r_i}$ (resp. $\psi_{r_i}$) as
  given by \Cref{lem:phi train track} (resp. \Cref{lem:psi train
  track}). Thus we have an unfolding sequence 
  \[
    \begin{tikzcd} 
      \cdots 
      \arrow[r]  
      & \tau_{i+1} \arrow[r, "\phi_{r_{i+1}}"] 
      & \tau_i \arrow[r, "\phi_{r_i}"] 
      & \tau_{r_{i-1}} \arrow[r] 
      & \cdots \arrow[r, "\phi_{r_3}"] 
      & \tau_2 \arrow[r, "\phi_{r_2}"] 
      & \tau_1 \arrow[r, "\phi_{r_1}"] 
      & \tau_0,
    \end{tikzcd}
  \]
  and a folding sequence 
  \[
    \begin{tikzcd} 
      \cdots  
      & \tau_{i+1}' \arrow[l]  
      & \tau_i' \arrow{l}[above,pos=.4]{\psi_{r_{i+1}}} 
      & \tau_{i-1}' \arrow{l}[above,pos=.4]{\psi_{r_i}} 
      & \cdots \arrow[l] 
      & \tau_2' \arrow{l}[above,pos=.4]{\psi_{r_3}}  
      & \tau_1' \arrow{l}[above,pos=.4]{\psi_{r_2}} 
      & \tau_0' \arrow{l}[above,pos=.4]{\psi_{r_1}}. 
    \end{tikzcd}
  \]

  Let $\Phi_i = \phi_{r_1} \circ \ldots \circ \phi_{r_i}$ and $\Phi_i^{-1}
  = \Psi_i = \psi_{r_i} \circ \ldots \circ \psi_{r_1}$. Here, $\tau_0$ is a
  rose with petals labeled by elements in $\{a,b,c,d,e,f,g\}$ and hence for
  $i \ge 1$, $\tau_i$ is a rose labeled by $\{\Phi_i(a), \ldots,
  \Phi_i(g)\}$. Also, $\tau_0'$ is a rose labeled by
  $\{a,b,c,d,e,f,g\}$, so $\tau_i'$ is also a rose labeled by
  $\{\Phi_i(a), \ldots, \Phi_i(g)\}$. Thus, for every $i\geq 0$,
  $\tau_i$ and $\tau_i'$ have the same marking but different train track
  structures. In other words, they belong to the same simplex in $\CV[7]$.  

  The next lemma studies the behavior of illegal turns in a path along the
  folding sequence. This will be used in the proof of
  \Cref{prop:nongeometric} to show that the limit tree of the folding
  sequence is non-geometric. 

  \begin{lemma}\label{lem:lose illegal turns}
    
    Let $(r_i)_{i \ge 1}$ be strictly increasing such that $r_i \equiv 0$
    mod 3 and $r_1 > R$, where $R$ is the constant from \Cref{lem:lose
    illegal turns iwip}. Let $(\tau_i')_i$ be the corresponding folding
    sequence. Then for any edge path $\beta$ in $\tau'_j$ with at least one
    illegal turn, the number of illegal turns in
    $[\psi_{r_{j+3}}\psi_{r_{j+2}}\psi_{r_{j+1}}(\beta)]$ is less than the
    number of illegal turns in $\beta$. 

  \end{lemma}

  \begin{proof} 
  
    By \Cref{lem:psi train track}, the illegal turns in $\tau'_{j}$ are 
    $$\{a,e\}, \{a,\bar{g}\},\{e,\bar{g}\},
    \{b,\bar{d}\},\{c,\bar{b}\},\{d,\bar{c}\},\{f,\bar{e}\},\{g,\bar{f}\} $$
    and we have
    \begin{flalign*} 
      \{a,e\} & \xrightarrow{\psi_{r_{j+1}}} \{d,\bar{c}\}
      \xrightarrow{\psi_{r_{j+2}}} \{g,\bar{f}\} \\
      \{a,\bar{g}\} & \xrightarrow{\psi_{r_{j+1}}} \{d,\bar{c}\}
      \xrightarrow{\psi_{r_{j+2}}} \{g,\bar{f}\} \\
      \{b,\bar{d}\} & \xrightarrow{\psi_{r_{j+1}}} \{e,\bar{g}\}  \\
      \{c,\bar{b}\} & \xrightarrow{\psi_{r_{j+1}}} \{f,\bar{e}\}  \\
      \{d,\bar{c}\} & \xrightarrow{\psi_{r_{j+1}}} \{g,\bar{f}\} 
    \end{flalign*}
    
    Thus, for any illegal edge path $\beta \subset \tau_{j}'$, one
    of $\beta, \psi_{r_{j+1}}(\beta), \psi_{r_{j+2}}\psi_{r_{j+1}}(\beta)$
    has an illegal turn $\{x, y\}$ where $x, y \in \{e, f, g,
    \bar{e}, \bar{f}, \bar{g}\}$. 

    Consider the automorphism $\vartheta$ and corresponding train track map
    $h \from  \mathrm{R}_3 \to \mathrm{R}_3$ as in \Cref{lem:vartheta train
    track}. Then $h$ does not have any periodic INPs. Since $R$ is the
    constant from \Cref{lem:lose illegal turns iwip}, we get that 
    one of $[\psi_{r_{j+1}}(\beta)], [\psi_{r_{j+2}}\psi_{r_{j+1}}(\beta)],
    [\psi_{r_{j+3}}\psi_{r_{j+2}}\psi_{r_{j+1}}(\beta)]$ has fewer illegal
    turns than $\beta$. \qedhere 
  
  \end{proof}

\section{Limiting tree of folding sequence}
  
  \label{sec:arational tree}

  In this section, we will show that for appropriate choices of $(r_i)_i$,
  the projection of the folding sequences $(\tau_i')_i$ to the free factor
  complex $\FreeF[7]$ is a quasi-geodesic and hence converges to the
  equivalence class of an arational tree. We will also show that this tree
  is non-geometric.

  \subsection{Sequence of free factors}
  
  Given a sequence $(r_i)_{i \ge 1}$, recall that $\Phi_i =
  \phi_{r_1}\phi_{r_1}\cdots\phi_{r_i}$, where $\phi_r = \rho \phi^r$. For
  convenience, also set $\Phi_0 = \id$. We have the folding sequence 
  \[
    \begin{tikzcd} 
      \cdots  
      & \tau_{i+1}' \arrow[l]  
      & \tau_i' \arrow{l}[above,pos=.4]{\psi_{r_{i+1}}} 
      & \tau_{i-1}' \arrow{l}[above,pos=.4]{\psi_{r_i}} 
      & \cdots \arrow[l] 
      & \tau_2' \arrow{l}[above,pos=.4]{\psi_{r_3}}  
      & \tau_1' \arrow{l}[above,pos=.4]{\psi_{r_2}} 
      & \tau_0' \arrow{l}[above,pos=.4]{\psi_{r_1}}. 
    \end{tikzcd}
  \]
  where $\tau_i'$ is a rose labeled by $\{ \Phi_i(a),\cdots,\Phi_i(g)\}$,
  and $\psi_r = \phi_r^{-1}$. From
  the markings, we can associate $\tau_i'$ to an open simplex $U_i$ in
  $\CV[7]$. Consider a sequence of free factors $A_i \in \pi(U_i)$, where
  $\pi \from  \CV[7] \to \FreeF[7]$. For an appropriate sequence of $(r_i)_i$,
  we will see that $(A_i)_i$ is a quasi-geodesic (with infinite diameter).
  The key will be \Cref{lem:consecutive} which is the main goal of
  this section. 

  We now consider the following explicit sequence of free factors. Let $A_0
  = \la d,e,f\ra$ be the free factor in $\free[7]$, and define $$ A_i:=
  \Phi_i (A_0) = \la \Phi_i(d), \Phi_i(e), \Phi_i(f)\ra.$$ Note that for any
  $r,s,t >0$, the following holds:
  \begin{equation}\label{eqn:ff1} 
    \begin{aligned} 
      A_0 &= \la d, e, f \ra \\ 
      A_1 &= \phi_r(A_0) = \la a, b, c \ra \\ 
      A_2 &= \phi_s\phi_r(A_0) = \la e, f, g \ra \\ 
      A_3 &= \phi_t\phi_s\phi_r(A_0) = \la b,c,d \ra 
    \end{aligned} 
  \end{equation} 
  Thus, for any sequence $(r_i)_i$,
  \begin{equation}\label{eqn:ff2} 
    A_i = \Phi_i(A_0) = \Phi_{i-1}(A_1) = \Phi_{i-2}(A_2) =
    \Phi_{i-3}(A_3). 
  \end{equation}

  We say two free factors $A$ and $A'$ are \emph{disjoint} if (possibly
  after conjugating) $\free=A*A'*B$ for a (possibly trivial) free factor
  $B$, and $A'$ is \emph{compatible} with $A$ if it either contains $A$ (up
  to conjugation) or is disjoint from $A$. 

  \begin{lemma}\label{lem:i-j}

    For any sequence $(r_i)_{i \ge 1}$, if $|i-j|=1$, then $A_i, A_j$ are
    disjoint; and if $|i-j| = 2$ or $3$, then they are distinct and not
    disjoint. 

  \end{lemma}

  \begin{proof}

    We see from \Cref{eqn:ff1} that the statement of the lemma
    holds for $A_0, A_1, A_2$ and $A_3$. Now for $i \ge 1$ and $k \in
    \{1,2,3\}$, by \Cref{eqn:ff2}, the pair $A_i, A_{i+k}$
    differs from $A_0, A_k$ by the automorphism $\Phi_i$, whence the lemma.
    \qedhere 
  \end{proof}

  Recall the transition matrix $M_r$ for $\phi_r$, and the $3\times 3$
  matrix $B$ whose power $B^r$ forms a block of $M_r$. For each $i \ge 1$,
  let $\overline{M}_i = M_i$ mod 2.  By a simple computation, we see that
  $B^7 = I$ mod 2. Thus, when $i = j$ mod 7, $\overline{M}_i =
  \overline{M}_j$. We have the following lemma.

  \begin{lemma}\label{lem:sage fact} 
    
    Let $V_0$ be the 3-dimensional vector
    space of $(\mathbb{Z}/2\mathbb{Z})^7$ spanned by the vectors
    $(0,0,0,1,0,0,0)^T,(0,0,0,0,1,0,0)^T,(0,0,0,0,0,1,0)^T$. Then for all
    $i \ge 0$, 
    \[
      V_0 \bigcup \left( \bigcup_{j=0}^{107} \overline{M}_i
      \overline{M}_{i+1}\ldots \overline{M}_{i+j} V_0 \right) =
      (\mathbb{Z}/2\mathbb{Z})^7.
    \]
  
  \end{lemma}

  \begin{proof} 
    
    Since $\overline{M}_i=\overline{M}_j$ whenever $i = j$ mod
    7, it is enough to verify the statement for $i \in \{0,\ldots,6\}$. In
    these cases, we can check the validity of the statement using {\tt Sage}
    with the following code: \lstset{language=Python} \lstset{frame=lines}
    \lstset{label={lst:code_direct}}
    \begin{lstlisting} 
B = matrix(GF(2), [
        [0,0,1],
        [1,0,0],
        [0,1,1]
    ])

def M(i):
    return block_matrix([
        [ matrix(4,3,0) , identity_matrix(4) ],
        [ B^i           , matrix(3,4,0)      ]
    ])

V0 = set( 
    (0,0,0,i,j,k,0) 
    for i in (0,1) 
    for j in (0,1) 
    for k in (0,1)
)

for i in range(0,7):
    W = set(V0)
    P = identity_matrix(7)
    for j in range(i,200):
        P = P*M(j)
        for v in V0:
            w = tuple(P*vector(v))
            W.add(w)
        if len(W) >= 2^7:
            break
    print(i,j)

# Output:
# 0 107
# 1 107
# 2 107
# 3 107
# 4 107
# 5 107
# 6 107\end{lstlisting}  

  \end{proof}

  \begin{lemma}\label{lem:consecutive} 
    For any sequence $(r_i)_i$, if $r_i \equiv i \mod 7$, then $109$
    consecutive $A_i$'s cannot be contained in the same free factor or be
    disjoint from a common factor.
  \end{lemma}

  \begin{proof}
    
    For any $i \ge 1$ and $k \ge 0$, let \[B_{i+k} = \phi_i \phi_{i+1}
    \cdots \phi_{i+k} A_0. \] Abelianizing and reducing mod 2, we have $A_0
    \equiv V_0$, and $B_{i+k} \equiv \overline{M}_i \cdots
    \overline{M}_{i+k} V_0$. Thus, by \Cref{lem:sage fact}, the
    sequence $\{A_0, B_i, \ldots, B_{i+107}\}$ cannot be contained in the
    same free factor or be disjoint from a common factor.
    
    Now consider any sequence $(r_i)_i$ with $r_i \equiv i$ mod $7$, so
    that $\overline{M}_{r_i} = \overline{M}_i$ for all $i$. Let $A_i =
    \Phi_i A_0 = \phi_{r_1} \cdots \phi_{r_i}A_0$. Set $\Phi_0 = \id$. For
    any $i \ge 1$, by applying the automorphism $\Phi_{i-1}^{-1}$, the
    sequence of free factors $\{A_{i-1},\ldots,A_{i+107}\}$ is isomorphic to
    the sequence \[\{A_0, \phi_{r_i}A_0, \ldots, \phi_{r_i}\phi_{r_{i+1}}
    \cdots \phi_{r_{i+107}}A_0 \}. \] The latter sequence after
    abelianization and reducing mod 2 is equivalent to the sequence
    $\{A_0,B_i,\ldots,B_{i+107}\}$. Thus $\{A_{i-1},\ldots,A_{i+107}\}$
    cannot be contained in the same factor or be disjoint from a common
    factor.\qedhere
  \end{proof}

  \subsection{Subfactor projection}

  We will now use subfactor projection theory originally introduced in
  \cite{BF:Subfactor} and further developed in \cite{Taylor:Subfactor}
  to show that $(A_i)_i$ is  a quasi-geodesic for
  appropriate choices of sequence $(r_i)_i$.

  We first define subfactor projection and recall the main results about
  them. For $G \in \CV$ and a rank $\geq 2$ free factor $A$, let $A|G$
  denote the core subgraph of the cover of $G$ corresponding to the
  conjugacy class of $A$. Pulling back the metric on $G$, we obtain $A|G
  \in \mathrm{CV}(A)$. Denote by $\pi_A(G):= \pi(A|G) \subset \F(A)$ the
  projection of $A|G$ to $\F(A)$. Here $\mathrm{CV}(A)$ is the Outer space
  of the free group $A$ and $\F(A)$ is the corresponding free factor
  complex. 

  Recall two free factors $A$ and $B$ are \emph{disjoint} if they are
  distinct vertex stabilizers of a free splitting of $\free$. If $B$ is not
  compatible with $A$, then we say $B$ \emph{meets} $A$, that is, $B$ and
  $A$ are not disjoint and $A$ is not contained in $B$, up to conjugation.
  In this case, define the projection of $B$ to $\F(A)$ as follows:
  \[\pi_A(B):=\{\pi_A(G) | G \in \CV \text{ and } B|G \subset G \text{ is
  embedded }\}\] If $B$ is compatible with $A$, then define $\pi_A(B)$ to
  be empty. If $A$ meets $B$ and $B$ meets $A$, then we say $A$ and $B$
  \emph{overlap}.  
  
  \begin{thm}[\cite{Taylor:Subfactor}]\label{thm:subfactor projection}

    Let $A,B,C$ be free factors of $\free$. There is a constant $D$ 
    depending only on $n$ such that the following statements hold.
    \begin{enumerate} 
      \item If $\rank(A) \ge 2$, then either $A \subseteq B$ (up to
        conjugation), $A$ and $B$ are disjoint, or $\pi_A(B) \subset
        \mathcal{F}(A)$ is defined and has diameter $\leq D$. 
      \item If $\rank(A) \ge 2$, $B$ and $C$ meet $A$, and $B$ is
        compatible with $C$, then \[ d_A(B,C) = \diam_{\F(A)}(\pi_A(B) \cup
        \pi_A(C)) \le D.\] 
      \item If $A$ and $B$ overlap, have rank at least $2$, and $C$ meets
        both, then \[ \min \{ d_A(B,C), d_B(A,C)\} \le D.\] 
    \end{enumerate}
  \end{thm}

  \begin{thm}[Bounded geodesic image theorem
    \cite{Taylor:Subfactor}]\label{thm:BGI} For $n \geq 3$, there exists $D'
    \geq 0$ such that if $A$ is a free factor with $\text{rank}(A) \geq 2$
    and $\gamma$ is a geodesic of $\FreeF$ with each vertex of $\gamma$
  having a well-defined projection to $\F(A)$, then
  $\text{diam}(\pi_A(\gamma)) \leq D'$. \end{thm}

  We now prove the following lemma. 
  
  \begin{lemma}\label{lem:i-k projection}
    For any $K>0$, there exists a constant $r = r(K)$ such that for any
    sequence $(r_i)_{i \ge 1}$, if $r_i \ge r$ for all $i$, then the
    following statements hold:
    \begin{enumerate}
      \item For any $j \geq 2$, the projections of $A_{j-2}$ and $A_{j+2}$
        to the free factor complex $\F(A_j)$ are defined and the distance
        between them is at least $K$. 
      \item Let $D$ be the constant of \Cref{thm:subfactor
        projection}. If $K > 3D,$ then for any $i < j <k$, if
        $j-i\geq 2$ and $k-j\geq 2$, the projections of $A_i$ and $A_k$ to
        $\F(A_j)$ are defined and have distance at least $K-2D$.
    \end{enumerate}
  \end{lemma}

  \begin{proof}
    
    Recall for any $r$, $\phi_r = \rho \phi^r$, where $\phi$ restricts to a
    fully irreducible outer automorphism of $\la a, b, c\ra$. In
    particular, $\phi$ acts as a loxodromic isometry of the free factor
    complex $\F(\la a,b,c \ra)$, Thus, for any $K$, there exists $r=r(K)$
    such that for all $s \ge r$, the distance between $\phi^s(\la b,c\ra)$
    is at least $K+2D$ away from $\la a,b\ra$ in $\F(\la a,b,c\ra)$. 

    Now consider any sequence $(r_i)_i$ with $r_i \ge r$ for all $i$. By
    \Cref{lem:i-j} and \Cref{thm:subfactor projection}, the projections of
    $A_{j-2}$ and $A_{j+2}$ to $\F(A_j)$ are defined. Moreover, by
    \Cref{eqn:ff2}, we see that, by applying an automorphism, the distance
    between projections of $A_{j-2}$ and $A_{j+2}$ in $\F(A_j)$ is the same
    as the distance between the projections of $A_0 = \la d, e, f \ra$ and
    $\phi_{r_{j-1}}(A_3) = \la \phi_{r_{j-1}}(b),\phi_{r_{j-1}}(c),a \ra$
    to $\F(A_2) = \F(\la e, f, g \ra)$. Note that the rotation $\rho$ sends
    the free factor $\la a,b,c\ra$ to $A_2$, thus inducing an isometry from
    $\F(\la a,b,c\ra)$ to $\F(A_2)$. The projection of $A_0$ to $\F(A_2)$
    is $D$-close to the factor $\la e,f\ra = \rho (\la a,b \ra)$, and the
    projection of $\phi_{r_{j-1}}(A_3)$ to $\F(A_2)$ is $D$-close to the
    factor $\rho \phi^{r_{j-1}}(\la b,c \ra)$. Thus, the distance in
    $\F(A_2)$ of the two projections is at least $K$. This shows the first
    statement of the Lemma.
  
    Now fix $K > 3D$ and let $(r_i)_i$ be any sequence with $r_i \ge r(K)$
    for all $i$. We will prove the second statement by inducting on $l=k-i$
    with the previous statement giving the base case $l = 4$. Suppose we
    are given $A_i,A_j,A_k$ with $l=k-i>4$, $j-i,k-j\geq 2$. We first claim
    that projections of $A_{j+2},A_{j+3},\cdots, A_k$ to $\F(A_j)$ are
    defined, i.e.\ none of them are equal to or disjoint from $A_j$. For
    suppose $A_s$ is the first on the list that is equal to or disjoint
    from $A_j$. By \Cref{lem:i-j} we have $4\leq s-j < k-i$. By induction,
    the projections of both $A_j$ and $A_s$ to $\F(A_{j+2})$ are defined
    and the distance between their projections is $\geq K-2D > D$. Using
    statement 2 of \Cref{thm:subfactor projection}, this implies that $A_s$
    and $A_j$ cannot coincide or be disjoint, proving the claim. By the
    same argument, we also have that the projections of
    $A_i,A_{i+1},\cdots,A_{j-2}$ to $\F(A_j)$ are all defined.
    
    By the first statement of the lemma, we have $d_{A_j}(A_{j-2},A_{j+2})
    \ge K$. We now claim that $d_{A_j}(A_{j+2}, A_k) \le D$. If $k=j+3$, then
    $A_{j+2}$ and $A_k$ are disjoint, and the claim holds by statement 2 of
    \Cref{thm:subfactor projection}. If $k \ge j+4$, then applying
    induction again to $j$, $j+2$ and $k$, we see that $A_j$ and $A_k$ have
    well-defined projections to $\F(A_{j+2})$ and $d_{A_{j+2}}(A_j,A_k) \ge
    K-2D > D$. Now, the claim follows by the third statement of
    \Cref{thm:subfactor projection}. By the same argument, we also see that
    $d_{A_j}(A_i,A_{j-2}) \le D$. We now conclude $d_{A_j}(A_i,A_k) \ge
    K-2D$ by the triangle inequality. \qedhere
  \end{proof}

  We are now ready to prove the main results of this section.

  \begin{prop}\label{prop:qi-geodesic} 

    There exists $R>0$ such for any sequence $(r_i)_{i \ge 1}$, if $r_i \ge
    R$, and $r_i \equiv i$ mod 7, then the sequence $(A_i)_{i \ge 0}$ is a
    quasi-geodesic in $\FreeF[7]$. 

  \end{prop}

  \begin{proof} 

    Let $D$ be the constant of \Cref{thm:subfactor projection} and let $D'$
    be the constant of \Cref{thm:BGI}. Fix $K=4D+D'$. Let $R=r(K)$ be the
    constant of \Cref{lem:i-k projection}. Let $(r_i)_{i \ge 1}$ be any
    sequence with $r_i \ge R$ and $r_i \equiv i$ mod 7 for all $i$. We will
    show that the sequence $(A_i)_i$ goes to infinity with linear speed.
    More precisely, we will show that for any $d>0$, if $k-i \ge 110d +4$,
    then $d_{\FreeF[7]}(A_i,A_k) \ge d$. Suppose not. Let $\gamma$ be a
    geodesic between $A_i$ and $A_k$ of length $< d$. 
    
    For every $j \in \{i+2,\ldots,k-2\}$, there exists a free factor in
    $\gamma$ that is compatible with $A_j$. Indeed, if every free factor in
    $\gamma$ meets $A_j$, then by \Cref{thm:BGI}, projection of $\gamma$ to
    $A_j$ will be well-defined and has diameter bounded by $D'$. However,
    by \Cref{lem:i-k projection}, the projections of $A_i$ and $A_k$ to
    $\F(A_j)$ has distance at least $K - 2D > D'$. 
    
    By the pigeonhole principle, there exists a vertex $B$ of $\gamma$
    compatible with at least 110 free factors among
    $\{A_{i+2},\ldots,A_{k-2}\}$. By \Cref{lem:consecutive}, it is not
    possible for $B$ to be compatible with 109 consecutive $A_j$'s.
    Therefore, it must be possible to find $i+2 \le i' < j' < k' \le k-2$
    with $j' - i' \ge 2$ and $k' - j' \ge 2$, such that $B$ is compatible
    with $A_{i'}$ and $A_{k'}$, but $B$ meets $A_{j'}$. In particular,
    $\pi_{A_{j'}}(B)$ is defined. By \Cref{lem:i-k projection}, $A_{i'}$,
    $A_{k'}$ also have well-defined projections to $\F(A_{j'})$ with
    $d_{A_{j'}}(A_{i'},A_{k'}) \ge K-2D > 2D$. On the other hand, since $B$
    is compatible with both $A_{i'}$ and $A_{k'}$, we have
    $d_{A_{j'}}(A_{i'},B) \le D$ and $d_{A_{j'}}(A_{k'},B)\le D$ by
    \Cref{thm:subfactor projection}. This is a contradiction, finishing the
    proof that $d_{\FreeF[7]}(A_i, A_k) \ge d$ for all $k-i \ge 110d+4$.
    \qedhere 
  
  \end{proof}
  
  Recall that $\FreeF$ is Gromov hyperbolic and that its Gromov boundary is
  the space of equivalence class of arational trees. Also recall we say a
  folding sequence $(G_i)_i$ \emph{converges to an arational tree} $T$, if
  $\pi(U_i)$ converges to $[T] \in \partial \FreeF$, where $U_i$ is the
  open simplex in in $\CV$ associated to $G_i$. We have the following
  corollary.

  \begin{cor} \label{cor:arational}

    Given any strictly increasing sequence $(r_i)_{i \geq 1}$ satisfying $r_i
    \equiv i \mod 7$, the folding sequence $(\tau_i')_i$ converges to
    an arational tree $T$. 
  \end{cor}

  \subsection{Non-geometric tree}

  \label{sec:trees}

  We will now show that the arational tree obtained in the previous
  section as the limit of the free factors $(A_i)_i$ is
  non-geometric. This section will use the terminology of band
  complexes and resolutions, for details see \cite{BF:Stable}.

  \begin{definition}[Geometric tree]
    \cite{BF:OuterLimits, levitt-paulin}\label{defn:geometric tree}
    Let $X$ be a band complex and $T$ a $G=\pi_1(X)$-tree. A resolution
    $f :\widetilde{X} \to T$ is \emph{exact} if for every $G$-tree
    $T'$ and equivariant factorization 
    \[
      \widetilde{X} \stackrel{f'}{\longrightarrow} T' 
      \stackrel{h}{\longrightarrow} T
    \] 
    of $f$ with $f'$ a surjective resolution it follows that $h$ is an
    isometry onto its image. We say $T$ is \emph{geometric} if every
    resolution is exact.
  \end{definition}

  The proof of the following proposition is based on \cite[Proposition
    3.6]{BF:OuterLimits}.

  \begin{prop}\label{prop:nongeometric}
    For any strictly increasing sequence $(r_i)_{i \ge 1}$, if the
    corresponding folding sequence $(\tau_i')_i$ converges to an arational
    tree $T$, then $T$ is not geometric. 
  \end{prop}

  \begin{proof}
    
    Let $\tilde\psi_i \from \tilde\tau_{i-1}' \to \tilde\tau_i'$ be a lift
    of the train track map to the universal covers fixing a base
    vertex. Pick a length measure on $(\tau_i')_i$ so we get a folding
    sequence $\tilde \tau_0' \stackrel{\tilde \psi_1}{\longrightarrow}
    \tilde\tau_1' \stackrel{\tilde \psi_2}{\longrightarrow} \cdots$ in
    $\cv[7]$ that converges to $T$.  Recall that there are morphisms
    $h_i \from \tilde{\tau}_i' \to T$ such that $h_i = h_{i+1} \tilde
    \psi_{i+1}$. Since $T$ is arational, $h_i$'s are not isometries though
    they restrict to isometries on edges. Let $X$ be a finite band
    complex with resolution $f \from \tilde X \to T$. We will show that
    the resolution factors through $\tilde\tau_i'$ for sufficiently
    large $i$. This will imply $T$ is not geometric.

    Let $\Gamma$ be the underlying real graph of $X$ (disjoint union
    of metric arcs) with preimage $\tilde\Gamma$ in $\tilde X$.
    We may assume $f$ embeds the components of $\tilde\Gamma$. A vertex $v$
    of $\tilde X$ is either a vertex of $\tilde\Gamma$ or a corner of a
    band or a 0-cell of $\tilde X$. For every such vertex $v$ choose a
    point $f_0(v) \in \tilde\tau_0$ so that $f_0$ is equivariant and $f=h_0
    f_0$ on the vertices of $\tilde X$.  

    An edge in $\tilde X$ is either a subarc of $\tilde\Gamma$ or a
    vertical boundary component of a band or a 1-cell in $\tilde X$. Up to
    the action of $\free[7]$, there are only finitely many edges. Using
    \Cref{lem:lose illegal turns}, we can find $i>0$ such that for every
    edge $e$ in $\tilde X$, the edge path in $\tilde\tau_i'$ joining the
    two vertices of $\tilde\psi_i \cdots \tilde\psi_1 f_0(\partial e)$ is
    legal. Now extend $\tilde\psi_i \cdots \tilde\psi_1 f_0$ to an
    equivariant map $f_i \from \tilde X \to \tilde\tau_i'$ that sends edges
    to legal paths (or points) and is constant on the leaves. Thus $f_i$ is
    a resolution of $\tilde\tau_i'$.
    \[\begin{tikzcd}
      &  & 
      \widetilde{X}\arrow[rd, bend left, "f_i"] \arrow[lld, bend right, "f_0"]
      \arrow[rrrd, bend left, "f"]\\ 
      \tilde{\tau}_0' \arrow[r, "\psi_1"] 
      & \tilde{\tau}_1' \arrow[r, "\psi_2"] 
      & \cdots \cdots  \arrow[r, "\psi_i"] 
      & \tilde{\tau}_i' \arrow[r, "\,\, \psi_{i+1}"] \arrow[rr, bend left, "h_i"]
      & \cdots 
      & T 
    \end{tikzcd}\]
    This yields a factorization
    \[\begin{tikzcd}
      \tilde X \arrow[r, "f_i"]  & \tilde\tau_i' \arrow[r, "h_i"] & T  
    \end{tikzcd}\]
    but $h_i$ is not an isometry. This shows $T$ is non-geometric. \qedhere 
  \end{proof}

\section{Non-uniquely ergodic unfolding sequence}

  \label{sec:unfolding}

  The goal of this section is to show that if a sequence $(r_i)_{i \ge 1}$
  grows sufficiently fast, then the set of currents supported on the legal
  lamination $\Lambda$ of the unfolding sequence $(\tau_i)_{i \geq 0}$
  is a 1-simplex in $\PCurr[7]$. 
  
  \medskip 
  
  Recall that $M_r$ is a $7\times 7$ matrix of the block form
  \[
    \begin{pmatrix}
      0&I\\
      B^r&0
    \end{pmatrix}
    \]
  where $I$ is the $4\times 4$ identity matrix, and $B$ is the transition
  matrix of $\theta$; all that matters is that some positive power of $B$
  has all entries positive. Let $\lambda_B$ be the Perron-Frobenius
  eigenvalue of $B$. Recall the constant $\kappa_B > 0$ from
  \Cref{lem:Y}. Given a sequence $(r_i)_i$, define for each $i \ge 1$ 
  \[
    P_i = \frac{1}{\kappa_B \lambda_B^{r_{i+1}}} M_{r_i} M_{r_{i+1}}.
  \]
  Let $\{e_k:k=1,\ldots,7\}$ be the standard basis for $\R^7$. Denote by
  $\mathbb P\R_{\geq 0}^7$ the projectivization of $\R_{\geq 0}^7$, and the
  projective class of a vector $v$ by $[v]$. Fix a metric $d$ on
  $\P\R_{\geq 0}^7$. We view $M_r$ as a projective transformation
  $\P\R_{\geq 0}^7\to \P\R_{\geq 0}^7$. For a sequence $(r_i)_{i \geq  1}$
  and for $i<j$ denote by $S_{i,j}\subset\mathbb P \R_{\geq 0}^7$ the image
  of the composition $$M_{ij}:=M_{r_i}M_{r_{i+1}}\cdots M_{r_j}$$ and by
  $S_i=\bigcap_{j>i}S_{i,j}$. We denote by $v_B$ a positive
  Perron-Frobenius eigenvector of $B$, and by $v_B^{234}$ (resp.\ 
  $v_B^{567}$) the vector in $\R^7$ which is $v_B$ in coordinates $2,3,4$
  (resp. $5,6,7$) and 0 in all other coordinates. The main result of this
  section is the following.

  \begin{prop}\label{prop:2currents}
    Let $(r_i)_{i \ge 1}$ be a sequence of positive integers with
    $r_{i+1}-r_i \ge i$. Then for all $i$ the set $S_i$ is a 1-simplex,
    i.e.\ it is the convex hull of two distinct points $p_i,q_i\in\mathbb P
    \R_{> 0}^7$. Moreover, as $i\to\infty$, $\{p_i,q_i\}$ converges (as a
    set) to $\{ [v_B^{234}], [v_B^{567}]\}$. \end{prop}

  Before we give a technical proof of \Cref{prop:2currents}, we will give a
  simpler, more intuitive proof where the sequence $r_1<r_2<\cdots$ is
  chosen inductively so that $r_1$ is sufficiently large and each $r_i$ is
  sufficiently large depending on $r_1,r_2,\cdots,r_{i-1}$. Later, we do a
  more careful analysis where we can control the growth of the sequence.   
  
  \begin{proof}[Proof idea of \Cref{prop:2currents}]
    For $\epsilon>0$ we will write $x\overset{\epsilon}=y$ if
    $d(x,y)<\epsilon$ in $\P\R_{\geq 0}^7$.
    Each $S_{ij}$ is the convex hull of the $M_{ij}$-images of the vectors
    $e_i$, $i=1,\cdots,7$. The proof consists of computing these
    images using the Perron-Frobenius dynamics.
    We first observe that there is a sequence $\epsilon_r\to 0$ such
    that:
    \begin{itemize}
      \item $M_r(e_7)=e_4$, $M_r(e_6)=e_3$, $M_r(e_5)=e_2$,
        $M_r(e_4)=e_1$,
      \item $M_r(e_i)\overset{\epsilon_r}=v_B^{567}$, $i=1,2,3$,
        \item $M_r(v_B^{567})=v_B^{234}$,
          $M_r(v_B^{234})\overset{\epsilon_r}=v_B^{567}$.
    \end{itemize}

    Next, we consider the composition $M_sM_r$ for $r>>s$. The third
    bullet uses uniform continuity of $M_s$ and the assumption that
    $r$ is sufficiently large compared to $s$.
    \begin{itemize}
    \item $M_sM_r(e_7)=e_1$,
    \item $M_sM_r(e_i)\overset{\epsilon_s}=v_B^{567}$, $i=4,5,6$,
    \item $M_sM_r(e_i)\overset{\epsilon_s}=v_B^{234}$, $i=1,2,3$.
    \end{itemize}
    Finally, for $r>>s>>t$ we see similarly:
    \begin{itemize}
    \item $M_tM_sM_r(e_7)\overset{\epsilon_t}=v_B^{567}$,
      \item $M_tM_sM_r(e_i)\overset{\epsilon_t}=v_B^{234}$,
        $i=4,5,6$,
      \item $M_tM_sM_r(e_i)\overset{\epsilon_t}=v_B^{567}$,
        $i=1,2,3$.
    \end{itemize}

    It follows that if we make suitably large choices for the $r_i$'s, the
    set $S_{i,i+3}$ will be contained in the $\epsilon_{r_i}$-neighborhood
    of the 1-simplex $[v_B^{567},v_B^{234}]$. Moreover, given any
    $\epsilon>0$ and $j>i+3$ we can choose $r_j$ large (depending on
    uniform continuity constants of $M_{ij}$) to ensure that
    $S_{i,j+3}=M_{ij}(S_{j,j+3})$ is contained in the
    $\epsilon$-neighborhood of the 1-simplex with endpoints
    $M_{ij}(v_B^{567})$ and $M_{ij}(v_B^{234})$. Thus each $S_i$ is the
    nested intersection of simplices of dimension $\leq 6$ such that for
    all $\epsilon>0$ they are eventually all contained in the
    $\epsilon$-neighborhood of a 1-simplex with definite distance between
    the endpoints. This proves the Proposition.
  \end{proof}

  We now present a more detailed proof of \Cref{prop:2currents}. 
  For a sequence of integers $(r_i)_{i\geq 1}$ such that $r_i, r_{i+1}-r_i
  \to \infty$, by \Cref{lem:Y} $(P_i)_i$ converges to an idempotent matrix
  $Y$.  Let $\Delta_i = Y - P_i$ and let $||Y||$ be the operator norm. 

  \begin{lemma}\label{lem:error bound} 
    Let $(r_i)_{i\geq 1}$ be a sequence of positive integers such that
    $r_{i+1}-r_i \ge i$. Then there exists an $I \geq 1$, such that for all
    $i \geq I$, $||\Delta_i|| \leq 1/(2 \cdot 2^i)$. 
  \end{lemma}

  \begin{proof}
    Let $\lambda_B, \mu_B, \mu'_B$ be the modulus of the three
    eigenvalues of $B$; we have $\lambda_B \sim 1.46$ and $\mu_B = \mu'_B
    \sim 0.826$. Then \[||\Delta_i|| = || P_i - Y || \leq \text{ max }
    \left( \frac{\mu^{r_{i+1}}}{\lambda^{r_{i+1}}},
    \frac{\lambda^{r_{i}}}{\lambda^{r_{i+1}}} \right) \leq
    \frac{\lambda^{r_{i}}}{\lambda^{r_{i+1}}} \]
    where the two terms comes from the two blocks in $P_i$. For the last
    inequality, note that $\mu < 1 < \lambda$ and $r_i$ are positive
    integers. Therefore,  $\mu^{r_{i+1}} < 1 < \lambda^{r_i}$. 
   
    Now we claim that there exists an $I \geq 1$, such that for all $i\geq I$, 

   \[ \frac{\lambda^{r_{i}}}{\lambda^{r_{i+1}}} \leq \frac{1}{2^{i+1}}
    \qquad \text{ equivalently, } \qquad 2 \leq
    \lambda^{\frac{r_{i+1}-r_i}{i+1}}\]
   
    We only need to show that the sequence $\frac{r_{i+1}-r_i}{i+1}$ is
    eventually increasing. Indeed, by assumption, $r_{i+1}-r_i \ge i$, so 
    \begin{align*}
      i  & \leq r_{i+1}-r_i \\
      \frac{ i}{i+1}  & \leq \frac{r_{i+1}-r_i}{i+1} 
    \end{align*}
    Since $i/(i+1)$ is an increasing sequence, it follows that our sequence
    is also increasing. \qedhere

  \end{proof}

  The following lemma is a consequence of \Cref{lem:Y} and
  \Cref{lem:convergence}.
  
  \begin{lemma} \label{lem:pairprod1}

    Let $(r_i)_{i \ge 1}$ be a sequence of positive integers such that
    $r_{i+1}-r_i \ge i$, $Y$ be the idempotent matrix of \Cref{lem:Y} and
    $M_\infty = \lim_{r \to \infty}
    M_r/\lambda_B^r$. Then the following statements hold. 
    \begin{enumerate}[(1)]
      \item For all $i \ge 1$, the sequence of matrices $\{ P_i P_{i+2}
        \cdots P_{i+2k}\}_{k=1}^\infty$ converges to a matrix $Y_i$.
        Furthermore, for all sufficiently large $i$, 
        \[ \norm{Y_i - Y} \le \frac{2}{2^i}\left( \norm{Y}+\norm{Y}^2 \right)\]
      \item The kernel of $Y$ is a subspace of the kernel of $Y_i$ for all
        $i \ge 1$.
      \item
        For all $i \ge 1$, $Y_i(e_1) \neq 0$ with non-negative entries and $Y_i(e_2)$
        and $Y_i(e_3)$ are positive multiples of $Y_i(e_1)$. 
      \item For all $i \ge 1$, $M_{r_i} Y_{i+1}(e_1) \neq 0$ with non-negative entries,
        and $M_{r_i}Y_{i+1}(e_1)$ and $Y_i(e_1)$ are not scalar multiples of
        each other.
      \item Projectively, $[Y_i(e_1)] \to [Y(e_1)]$ and
        $[M_{r_i}Y_{i+1}(e_1)] \to [M_\infty Y(e_1)]$ as $i \to \infty$. 
    \end{enumerate}
  \end{lemma}

  \begin{proof}
    
    For (1), it suffices to show convergence for all $i$ greater than some
    $I$. Indeed, if such $I$ exists and $i < I$, then let $i_0 \ge I$ be
    such that $i = i_0$ (mod 2) and observe that
    \[ \{P_i P_{i+2} \cdots P_{i+2k}\}_{k=\frac{i_0-i}{2}}^\infty =
    P_i P_{i+2} \cdots P_{i_0-2} \{P_{i_0}P_{i_0+2} \cdots
    P_{i_0+2k}\}_{k=0}^\infty.\] By assumption $\{P_{i_0}P_{i_0+2} \cdots
    P_{i_0+2k}\}_{k=0}^\infty$ converges. Since matrix multiplication is
    continuous, the sequence $\{P_iP_{i+2} \cdots P_{i+2k}\}_{k=0}^\infty$
    also converges.

    For each $i$, let 
    \[ \Delta_i = P_i - Y.\]
    By \Cref{lem:error bound},
    there exists $I \ge 1$ such that for all $i \ge I$, $\norm{\Delta_i} \le
    \frac{1}{2 \cdot 2^i}$. Also, choose $I$ sufficiently large so that
    $\frac{1}{2^I} \norm{Y} \le 1/2$. Then, by \Cref{lem:convergence},
    for all $i \ge I$, the sequence $\{ P_i P_{i+2} \cdots
    P_{i+2k}\}_{k=0}^\infty$ converges to some matrix $Y_i$, with 
    \begin{equation} \label{eqn:bound}
      \norm{Y_i - Y} \le \frac{2}{2^i} \left( \norm{Y}+\norm{Y}^2 \right).
    \end{equation}

    For (2), it again suffices to show the statement is true for all
    sufficiently large $i$, and the statement holds for all $i \ge I$ by
    \Cref{lem:convergence}. 
    
    For (3), first note that since all the matrices involved are
    non-negative, the resulting vectors are all also non-negative. So we
    only need to show that they are not the zero vector. It suffices to
    check that $Y_i(e_1) \neq 0$ for all sufficiently large $i$, since each
    $P_i$ is non-negative and has full rank. For large $i$, the statement
    follows because $Y(e_1)$ is not equal to 0 and $\norm{Y_i(e_1) -
    Y(e_1)} \le \norm{Y_i-Y}$ can be made arbitrarily small. For the second
    statement, we know that $Y(e_2)$ and $Y(e_3)$ are positive multiples of
    $Y(e_1)$, so there are $s,t > 0$ such that $se_2 - e_1$ and $te_3 -
    e_1$ are in the kernel of $Y$. Then $Y_i(se_2-e_1) = Y_i(te_3-e_1)=0$
    for all $i$ by (2). 
    
    For (4), $M_{r_i}Y_{i+1}(e_1) \neq 0$ with non-negative entries since
    $Y_{i+1}(e_1)$ is so by (3). To see that $M_{r_i}Y_{i+1}(e_1)$ and
    $Y_i(e_1)$ are projectively distinct, it is enough to do this for all
    sufficiently large $i$. Let $M_\infty = \lim_{r \to \infty}
    M_r/\lambda_B^r$. By \Cref{lem:Y} and \Cref{lem:MYMZ}, $Y(e_1)$ and
    $M_\infty Y (e_1)$ are orthogonal. Since $r_i \to \infty$, we can make
    $\frac{M_{r_i}}{\lambda_B^{r_i}}Y_{i+1}(e_1)$ arbitrarily close to
    $M_\infty Y(e_1)$, and $Y_i(e_1)$ close to $Y(e_1)$. This means
    $M_{r_i} Y_{i+1}(e_1)$ and $Y_i(e_1)$ are near orthogonal, so they
    can't be scalar multiples of each other. 
    
    Statement (5) is clear. \qedhere
  \end{proof}
  
  \begin{proof}[Proof of \Cref{prop:2currents}]
    
    By \Cref{lem:Y} and \Cref{lem:MYMZ},
    $[v_B^{234}]=[Y(e_1)]$ and 
    $[v_B^{567}]=[M_\infty Y(e_1)]$.
    Using notation from \Cref{lem:pairprod1}, set \[p_i = [Y_i(e_1)] \qquad
    \text{and} \qquad q_i = [M_{r_i}Y_{i+1}(e_1)].\] By
    \Cref{lem:pairprod1}
    (3) - (5), 
    \begin{itemize}
      \item $p_i$ and $q_i$ are well defined and distinct.
      \item $p_i = [Y_i(e_k)]$, and $q_i = [M_{r_i}Y_{i+1}(e_k)]$, for $k=1,2,3$.
      \item $p_i \to [v_B^{234}]$ and $q_i \to [v_B^{567}]$.
      \item $[M_{r_i}(p_{i+1})] = q_i$ and $[M_{r_i}(q_{i+1})] = p_i$. 
    \end{itemize}

    Our goal is to show $S_i$ is the $1$-simplex spanned by $p_i$ and
    $q_i$. To do this, we consider $S_{ij}$, which is the convex hull of
    the $M_{ij}$-images of the vectors $e_k$, $k=1,\cdots,7$. That is, we
    have to show that $[M_{ij}(e_k)]$ is close to either $p_i$ or $q_i$ for
    each $k$. We first observe that for all $r,s > 0$: 
    \begin{itemize}
      \item $M_r(e_4)=e_1$, $M_r(e_5)=e_2$, $M_r(e_6)=e_3$,
        $M_r(e_7)=e_4$,
      \item $M_rM_s(e_7)=e_1$.
    \end{itemize}
    
    We may assume that $j-1 = i+2m$, so $M_{ij}$ breaks up into pairs,
    i.e.\ for all $k$, 
    \[ [M_{ij}(e_k)] = [P_i \cdots P_{j-1} (e_k)].\]
    
    Let $\epsilon > 0$ be arbitrary. Choose
    $\delta > 0$ such that for any vector $u \in \R_+^7$ and any $v \in
    \left\{ Y_i(e_k), \frac{M_{r_i}}{\lambda_B^{r_i}} Y_{i+1}(e_k) :
    k=1,2,3\right\}$, if $\norm{u-v} \le \delta$, then $d([u],[v]) \le
    \epsilon$. Now by \Cref{lem:pairprod1}, we can choose $J$ sufficiently
    large so that whenever $i+2m \ge J$, then
    \begin{itemize}
      \item $\norm{ P_i \cdots P_{i+2m} -Y_i} \le \delta$
      \item $\norm{ P_{i+1} \cdots P_{i+2m+1} -Y_{i+1}} \le
        \frac{\delta}{\norm{M_{r_i}/\lambda_B^{r_i}}}.$
    \end{itemize}
    Now we may assume that $j-3 \ge J$. Then, 
    \begin{itemize}
      \item For $k=1,2,3$, we have \[ \norm{P_i \cdots P_{j-1}(e_k) -
        Y_i(e_k)} \le 
        \delta \quad \Longrightarrow \quad d \Big( [ M_{ij} (e_k) ],
        p_i \Big) \le
        \epsilon. \]
     \item For $k=7$, we have $M_{ij}(e_7) = M_{i,j-2}(e_1)$, so
       $[M_{ij}(e_7)]=[P_i \cdots P_{j-3}(e_1)]$ is $\epsilon$-close to
        $[p_i]$ by the same reasoning as the previous bullet point.
      \item For $k=4,5,6$, $M_{ij}(e_k) = M_{i,j-1}(e_{k-3})$. In this
        case, we consider $\frac{M_{r_i}}{\lambda^{r_i}} P_{i+1} \cdots
        P_{j-2}(e_{k-3})$ and approximate it by
        $\frac{M_{r_i}}{\lambda^{r_i}} Y_{i+1}(e_{k-3})$, as follows: 
        \begin{align*}
          & \norm{\frac{M_{r_i}}{\lambda^{r_i}} P_{i+1} \cdots
          P_{j-2}(e_{k-3}) -
          \frac{M_{r_i}}{\lambda^{r_i}} Y_{i+1}(e_{k-3})} \\ 
          & \le \norm{\frac{M_{r_i}}{\lambda^{r_i}}} \norm{P_{i+1} \cdots
          P_{j-2} - Y_{i+1}}\\
          & \le \delta. 
        \end{align*}
        Thus, for $k=4,5,6$, $d \Big( [M_{ij}(e_k)], q_i \Big) \le
        \epsilon$.

    \end{itemize}
    We have shown that for any $\epsilon$, the vertices of the simplex
    $S_{i,j}$ come $\epsilon$-close to $p_i$ and $q_i$ for all sufficiently
    large $j$. Since $S_{i,j+1} \subset S_{i,j}$ and $S_i = \bigcap_{j>i}
    S_{i,j}$, it follows that $S_i$ must be the $1$-simplex spanned by
    $p_i$ and $q_i$. This proves the Proposition. \qedhere
  \end{proof}
  
  Recall the unfolding sequence $(\tau_i)_{i \ge 0}$ where $M_{r_i}$ is the
  transition matrix of the train track map $\phi_{r_i}: \tau_i \to
  \tau_{i-1}$. Let $\Lambda$ be the legal lamination of $(\tau_i)_{i \ge 0}$.  
  
  \begin{corollary}\label{cor:currents 1-simplex}
    If $(r_i)_{i \ge 1}$ is a positive sequence with  $r_{i+1}-r_i \ge i$,
    then $\P\C(\Lambda)$ is a 1-simplex. 
  \end{corollary}

  \begin{proof}
    
    In light of \Cref{thm:NPR currents}, it is enough to show
    $\P\C((\tau_i)_i)$ is a 1-simplex. For each $i \ge 0$, we have a
    well-defined projection \[ p_i\from \P\C((\tau_i)_i) \to \P\R_+^7
    \quad\text{given by}\quad p_i([(\mu_i)_i]) = [\mu_i].\] The image of
    the projection is $S_{i+1}$, which is always a $1$-simplex by
    \Cref{prop:2currents}. Therefore, $\P\C((\tau_i)_i)$ is a $1$-simplex.
    \qedhere

  \end{proof}
  
\section{Non-uniquely ergometric tree}

  \label{sec:folding}

  The goal of this section is to show that if a sequence $(r_i)_{i \ge 1}$
  grows sufficiently fast, then the set of projectivized length measures
  $\P\D((\tau_i')_i)$ on the folding sequence $(\tau_i')_i)$ is a
  1-simplex. By \Cref{prop:IsoLength}, if $(\tau_i')_i$ converges to an
  arational tree $T$, then $\P\D(T)$ is also a 1-simplex in
  $\partial\CV[7]$. 

  Recall that $N_r$ is a $7\times 7$ matrix of the block form
  \[
    \begin{pmatrix}
      0&C^r\\
      I&0
    \end{pmatrix}
    \]
  where $I$ is the $4\times 4$ identity matrix, and $C$ is the transition
  matrix of $\vartheta$. The transpose of $N_r$ has the same shape as $M_r$.
  Therefore, the same theory from \Cref{sec:unfolding} holds true.
  For brevity, we record only the essential statements that will be used
  later and omit all proofs from this section. 

  Let $\lambda_C$ be the Perron-Frobenius eigenvalue of $C$.   Let
  $\kappa_C$ be the constants of \Cref{lem:Z}. Given a sequence
  $(r_i)_i$, define for each $i \ge 1$ 
  \[
    Q_i = \frac{1}{\kappa_C \lambda_C^{r_{i+1}}} N_{r_{i+1}} N_{r_i}.
  \]

  \begin{lemma} \label{lem:pairprod2}
 
    Given a sequence $(r_i)_{i \ge 1}$ of positive integers such that
    $r_{i+1}-r_i \ge i$. Then for all $i \ge 1$, the sequence of matrices
    $\{ Q_{i+2k}\cdots Q_{i+2} Q_i\}_{k=0}^\infty$ converges to a matrix
    $Z_i$. Furthermore, $\lim_{i \to \infty} Z_i=Z$, where $Z$ is the
    idempotent matrix of \Cref{lem:Z}.

  \end{lemma}
  
  \begin{cor}\label{cor:trees 1-simplex}

    If $(r_i)_{i \ge 1}$ is a positive sequence with  $r_{i+1}-r_i \ge i$.
    Then $\P\D((\tau_i')_i)$, and hence $\P\D(T)$, is a 1-simplex.

  \end{cor}

\section{Non-uniquely ergodic tree}

  \label{sec:combination}

  In this section we relate the legal lamination $\Lambda$ associated to
  the unfolding sequence $(\tau_i)_i$ defined in \Cref{sec:unfolding} and
  the limiting tree $T$ of the folding sequence $(\tau_i')_i$ defined in
  \Cref{sec:arational tree}, to show that $T$ is not uniquely ergodic.

  Recall the automorphism $\Phi_i = \phi_{r_1} \circ \cdots \phi_{r_i}$,
  with $\Phi_0 = \id$. We also use $\Phi_i$ to denote the induced graph map
  from $\tau_i$ to $\tau_0$. If each $\tau_i$ and $\tau_i'$ as a marked
  graph is the rose labeled by $\{a_i, b_i, c_i, d_i, e_i, f_i, g_i\}$,
  then  $x_i$ is represented by $\Phi_i(x)$ for $x \in \{a,b,c,d,e,f,g\}$
  as a word in $\free[7] = \langle a,b,c,d,e,f,g\rangle = \pi_1(\tau_0) =
  \pi_1(\tau_0')$. We denote $x_0$ as above simply by $x$. 

  \begin{lemma}\label{lem:volume0}

    If $(r_i)_{i \ge 1}$ is positive, then for any length measure
    $(\lambda_i)_i \in \D((\tau'_i)_i)$, the $\lambda_i$-volume of
    $\tau_i'$ goes to $0$ as $i \to \infty$. 

  \end{lemma}
  
  \begin{proof}
    
    The composition
    $\psi_{r_i}\psi_{r_{i-1}}\psi_{r_{i-2}}:\tau_{i-3}'\to\tau_i'$ has
    the property that the preimage of every point of $\tau_i'$ consists of
    at least two (in fact, many more) points of $\tau_{i-3}'$ and so
    the $\lambda_i$-volume of $\tau_i'$ is at most half of the
    $\lambda_{i-3}$-volume of $\tau_{i-3}'$. \qedhere
  \end{proof}

  \begin{lemma}\label{lem:leaf of Lambda} 

    Suppose $(r_i)_{i \ge 1}$ is positive. Let $\Lambda$ be the legal
    lamination of the unfolding sequence $(\tau_i)_i$. Then every leaf in
    $\Lambda$ is obtained as a limit of a sequence $\{\Phi_i(w)\}_i$, where
    $w$ is a legal word in $\tau_0$ of length at most two in $\{a, b, c, d, e, f, g\}$
    and their inverses. Moreover, $w$ can be closed up to a legal loop
    which is a cyclic word of length $\leq 3$.

  \end{lemma}

  \begin{proof} 
    Let $l$ be a leaf of $\Lambda$ realized as a bi-infinite line in
    $\tau_0$ and let $s$ be any subsegment of $l$, with combinatorial edge
    length $\ell_s >0$ in $\tau_0$. By definition, for every $i$ there is a
    bi-infinite legal path $l_i$ in $\tau_i$ such that $l = \Phi_i(l_i)$.
    Let $i = i(s) \geq 0$ such that the edge length of $x_i$ in $\tau_0$
    under the graph map $\Phi_i$ is $\geq \ell_s$ for all $x \in \{a, b, c,
    d, e, f, g\}$. Thus, there is a segment $s_i$ of $l_i$ of combinatorial
    length at most two in $\{a_i, b_i , c_i, d_i, e_i, f_i, g_i\}$, such
    that $s \subset \Phi_i(s_i)$ (here $\Phi_i$ is a graph map). Now if
    $s_i = x_iy_i$ for $x, y \in \{a,b,c,d,e,f,g\}$, take $w=xy$. Thus we
    see that $\Phi_i (w)$ (here $\Phi_i$ is an automorphism) covers $s$ in
    $\tau_0$. Since this is true for any segment of $l$, we conclude the
    lemma by taking a nested sequence of subsegments of $l$ with edge
    length in $\tau_0$ going to infinity. The fact that legal paths of
    length $\leq 2$ can be closed up to legal loops of length $\leq 3$
    follows from the description of the train track in \Cref{lem:phi train
    track}. \qedhere
  \end{proof}

  Recall that if $(\tau_i')_i$ converges to an arational tree $T$, then we
  can identify $\D((\tau_i')_i)$ with $\D(T)$ by
  \Cref{prop:IsoLength}.

  \begin{lemma}\label{lem:T translation length}

    Suppose $(r_i)_{i \ge 1}$ is positive and that the folding sequence
    $(\tau_i')_i$ converges to an arational tree $T$. Let $w$ be any
    conjugacy class 
    in $\mathbb F_7$ represented by a cyclic word in $\{a,b,c,d,e,f,g\}$
    and their inverses and let $\lambda \in \D(T)$
    correspond to a length measure $(\lambda_i)_i \in \D((\tau_i')_i)$.
    Then \[ \lim_{i \to \infty} \norm{\Phi_i(w)}_{(T,\lambda)} = 0. \]  

  \end{lemma}

  \begin{proof}

    Under the isomorphism from $\D((\tau_i')_i) \to \D(T)$ that maps
    $(\lambda_i)_i \mapsto \lambda$, the sequence $(\tau_i',\lambda_i)
    \subset \cv[7]$ also converges to $(T,\lambda) \in \partial \cv[7]$.
    Thus, for any $x \in \free[7]$, 
    \[\norm{x}_{(T,\lambda)} = \lim_{i \to \infty}\norm{x}_{(\tau_i',
    \lambda_i)}.\] In fact, the sequence $\norm{x}_{(\tau_i', \lambda_i)}$
    is monotonically non-increasing. Recall that $\tau_0'$ as a marked
    graph is the rose labeled by $\{a,b,c,d,e,f,g\}$. Represent $w$ by a
    loop $c_w$ in $\tau_0'$. The graph $\tau_i'$ is the rose labeled by
    $\{\Phi_i(a),\ldots,\Phi_i(g)\}$. Thus, the loop $c_w$ in $\tau_i'$
    represents the conjugacy class $\Phi_{i}(w)$. This shows 
    \[\norm{\Phi_i(w)}_{(T,\lambda)}\leq
    \norm{\Phi_i(w)}_{(\tau_i',\lambda_i)}\leq
    \norm{w}_{\mbox{word}}\vol(\tau_i',\lambda_i)\]
    where $\norm{w}_{\mbox{word}}$ is the word length of $w$. By
    \Cref{lem:volume0} the last term goes to 0.
    \qedhere
  \end{proof}

  We now come to the main statement of this section.

  \begin{prop}\label{prop:dual lamination}

    Suppose $(r_i)_{i \ge 1}$ is positive and that the folding sequence
    $(\tau_i')_i$ converges to an arational tree $T$. Let $\tilde{\Lambda}$
    be the lamination corresponding to the legal lamination $\Lambda$ of
    the unfolding sequence$(\tau_i)_i$ and let $L(T)$ be the lamination
    dual to $T$. Then $\tilde{\Lambda} \subseteq L(T)$. In particular, if
    $T$ is non-geometric, then $\C(\Lambda) = \C(T)$. 

  \end{prop}
  
  \begin{proof}

    Recall by \Cref{lem:arational equivalence}, the lamination dual to an arational tree is independent of the
    length measure on the tree. So fix an arbitrary length
    measure $\lambda \in \D(T)$ on $T$. 
    
    Let $W_3$ be the set of legal loops of length at most three in
    $\{a, b, c, d, e, f, g\}$ and their inverses. By \Cref{lem:T
      translation length}, for every $\epsilon >0$, there exists
    $I_\epsilon >0$ such that for all $i \geq I_\epsilon$,
    $\norm{\Phi_i(w)_{(T,\lambda)}} < \epsilon$, for every $ w\in
    W_3$. Then the bi-infinite line $(\Phi_i(w)^{-\infty},
    \Phi_i(w)^{\infty})$ is in $L_{\epsilon}(T)$ for all $i \geq
    I_{\epsilon}$. Therefore, $$\bigcap_{\epsilon > 0}
    \overline{\bigcup_{\substack{w \in W_3 \\ i \geq I_{\epsilon}}}
      (\Phi_i(w)^{-\infty}, \Phi_i(w)^{\infty})} \subseteq
    \bigcap_{\epsilon > 0} L_{\epsilon}(T). $$ By \Cref{lem:leaf
      of Lambda}, we conclude that $\tilde{\Lambda} \subseteq L(T)$.

    If $T$ is non-geometric and arational, then it is freely indecomposable
    by \cite{R:ReducingSystems}. By \cite[Corollary~1.4]{CHR}, $\C(\Lambda)
    = \C(T)$. \qedhere

  \end{proof}

  The following is the consequence of \Cref{prop:dual
  lamination} and \Cref{cor:currents 1-simplex}. 

  \begin{cor} \label{cor:nue}

    For a positive sequence $(r_i)_{i \ge 1}$ of integers with $r_{i+1}-r_i
    \ge i$, if the folding sequence $(\tau_i')_i$ converges to a
    non-geometric arational tree $T$, then $\P\C(T)$ is a 1-simplex. In
    particular, $T$ is not uniquely ergodic. \qedhere 

  \end{cor}

\section{Non-convergence of unfolding sequence}

  \label{sec:limit unfolding}
  
  In this section, fix a sequence $(r_i)_{i \ge 1}$ such that  $r_{i+1}-r_i
  \ge i$. We will show that the corresponding unfolding sequence
  $(\tau_i)_i$ does not converge to a unique point in $\partial \CV[7]$. In
  fact, we will show in \Cref{cor:ergotrees} that it converges to a
  1-simplex in $\partial \CV[7]$.
  
  \medskip 
  
  Recall the folding and unfolding sequences $(\tau_i')_i$ and
  $(\tau_i)_i$, respectively, from \Cref{sec:construction}.
  \[
    \begin{tikzcd} 
      \cdots 
      \arrow[r]  
      & \tau_{i+1} \arrow[r, "M_{r_{i+1}}"] 
      & \tau_i \arrow[r, "M_{r_i}"] 
      & \tau_{r_{i-1}} \arrow[r] 
      & \cdots \arrow[r, "M_{r_3}"] 
      & \tau_2 \arrow[r, "M_{r_2}"] 
      & \tau_1 \arrow[r, "M_{r_1}"] 
      & \tau_0,
    \end{tikzcd}
  \]
  and
  \[
    \begin{tikzcd} 
      \cdots  
      & \tau_{i+1}' \arrow[l]  
      & \tau_i' \arrow{l}[above,pos=.4]{N_{r_{i+1}}} 
      & \tau_{i-1}' \arrow{l}[above,pos=.4]{N_{r_i }} 
      & \cdots \arrow[l] 
      & \tau_2' \arrow{l}[above,pos=.4]{N_{r_3}}  
      & \tau_1' \arrow{l}[above,pos=.4]{N_{r_2}} 
      & \tau_0' \arrow{l}[above,pos=.4]{N_{r_1}}. 
    \end{tikzcd}
  \]
  Here $\tau_i$ and $\tau_i'$ as marked graphs belong to the same
  simplex in $\CV[7]$. Also, recall the matrices defined for all $i\ge 0$ 
  \[P_i = \frac{1}{\kappa_B \lambda_B^{r_{i+1} }} M_{r_i} M_{r_{i+1}}
  \qquad
  \text{and}
  \qquad
  Q_i = \frac{1}{\kappa_C \lambda_C^{r_{i+1} }} N_{r_{i+1}}N_{r_i},\]  
  and the existence of the limiting matrices from \Cref{lem:pairprod1} and
  \Cref{lem:pairprod2} 
  \[Y_i = \lim_{k \to \infty}P_iP_{i+2} \cdots P_{i+2k} 
  \qquad \text{and} \qquad Z_i = \lim_{k \to \infty}Q_{i+2k} \cdots
  Q_{i+2} Q_i.
  \] 
  For all even $2m \ge 0$, 
  \[ c_{2m} = \Big( \kappa_B^m \lambda_B^{r_2} \lambda_B^{r_4} \cdots
  \lambda_B^{r_{2m}} \Big) \Big( \kappa_C^m \lambda_C^{r_2} \lambda_C^{r_4}
  \cdots \lambda_C^{r_{2m}} \Big).
  \] 
  Similarly, for all odd $2m+1 \ge 1$, set 
  \[ c_{2m+1} = \Big( \kappa_B^m \lambda_B^{r_1}\lambda_B^{r_3} \cdots
  \lambda_B^{r_{2m+1}} \Big) \Big( \kappa_C^m \lambda_C^{r_1}
  \lambda_C^{r_3} \cdots \lambda_C^{r_{2m+1}} \Big).
  \] 
  Let $\ell=\ell_0 \in \R^{|E\tau_0|}$ be a positive length vector on
  $\tau_0$. Then $\ell$ determines a length vector $\ell_i$ on each
  $\tau_i$ given by $\ell_i = M_{r_i}^T \ldots M_{r_1}^T \ell \in
  \R^{|E\tau_i|}$. We set $\ell_e^T = \ell^T Y_1$ and $\ell_o^T = \ell^T
  \frac{M_{r_1}}{\lambda_B^{r_1}}Y_2$. Note that both $\ell_e$ and $\ell_o$
  are positive vectors. For $\ell_e$, this follows since $\ell$ is a
  positive vector and $Y_1$ is a non-negative matrix. Similarly, $\ell^T
  M_{r_1}$ is positive and $Y_2$ is non-negative, so $\ell_o$ is also
  positive.
  
  We will show the sequence $(\tau_i, \ell_i)_i \subset \CV[7] $, up to
  rescaling, does not have a unique limit in $\partial \CV[7]$. We start by
  showing the even sequence and the odd sequence do converge, up to
  scaling. More precisely:

  \begin{lemma}\label{lem:TeTo}
    
    For any positive length vector $\ell=\ell_0$ on $\tau_0$, the corresponding
    even sequence $\left( \tau_{2m}, \frac{\ell_{2m}}{c_{2m}} \right)$ 
    and odd sequence $\left( \tau_{2m+1}, \frac{\ell_{2m+1}}{c_{2m+1}} \right)$
    of metric graphs converge to two points $T_e$ and $T_o$ respectively
    in $\partial \CV[7]$. In fact, for any conjugacy class $x \in
    \free[7]$, there exists an index $i_x \ge 0$, a vector $v_x \in
    \R^{|E\tau_{i_x}'|}$, and matrices $Y_x^e$ and $Y_x^o$, such that 
    \[ 
    \norm{x}_{T_e} = \ell_e^T Y_x^e v_x 
    \qquad \text{and} \qquad
    \norm{x}_{T_o} = \ell_o^T Y_x^o v_x 
    \]

  \end{lemma}

  \begin{proof}
    Let $x \in \free[7]$ be a cyclically reduced representative of its
    conjugacy class. By \Cref{lem:lose illegal turns}, there exists $i
    \geq 0$ such that $x$ is legal in $\tau_{i}'$. Let $i_x$ be the
    smallest index among such $i$. Then we can represent $x$ by a vector
    $v_x$ in $\R^{|E\tau_{i_x}'|}$ and by the vector
    $N_{r_i}\ldots N_{{i_x+1}} v_{x}$ in $\R^{|E\tau_i|}$ for $i
    \geq i_x$. Thus, for all $i \ge i_x$, we have
    \[
      \norm{x}_{(\tau_i,\ell_i)} = \Big( \ell^T  M_{r_1} \cdots
      M_{r_i} \Big) \Big( N_{r_i} \cdots N_{{i_x+1}} v_x
      \Big).
    \]
    If $i_x$ is even, then write $i_x = 2m_x$, and set 
    \[ 
    c_x^e = \kappa_C^{m_x} \lambda_C^{r_2} \lambda_C^{r_4} \cdots
    \lambda_C^{r_{i_x}}
    \qquad \text{and} \qquad 
    c_x^o = \kappa_C^{m_x} \lambda_C^{r_1} \lambda_C^{r_3} \cdots
    \lambda_C^{r_{i_x-1}}. \] 
    If $i_x$ is odd, then write $i_x = 2m_x+1$, and set 
    \[ 
    c_x^e = \kappa_C^{m_x+1} \lambda_C^{r_2} \lambda_C^{r_4}\cdots
    \lambda_C^{r_{i_x-1}}
    \qquad \text{and} \qquad 
    c_x^o = \kappa_C^{m_x} \lambda_C^{r_1} \lambda_C^{r_3} \cdots
    \lambda_C^{r_{i_x}}. \] 
    First suppose $i_x$ is even. Then for all even $2m \ge i_x$, we have
    \begin{align*}
      \norm{x}_{ \left( \tau_{2m}, \frac{\ell_{2m}}{c_{2m}}\right)} 
      & = \frac{\norm{x}_{(\tau_{2m}, \ell_{2m})}}{c_{2m}} \\[.5em]
      & = \frac{\ell^T
      \Big(P_1 P_3 \cdots P_{2m-1} \Big) \Big( Q_{2m-1} \cdots
      Q_{i_x+3} Q_{i_x+1} \Big) v_x}{c_x^e} \\[.5em]
      &\xrightarrow{m \to \infty} \frac{\ell^T Y_1 Z_{i_x+1}
      v_x}{c_x^e}
      = \ell_e^T \left( \frac{Z_{i_x+1}}{c_x^e}\right) v_x
    \end{align*}
     and for odd $2m+1 \ge i_x$, we have
    \begin{align*}
      \norm{x}_{ \left( \tau_{2m+1}, \frac{\ell_{2m+1}}{c_{2m+1}}\right)} 
      & = \frac{\norm{x}_{(\tau_{2m+1}, \ell_{2m+1})}}{c_{2m+1}} \\[.5em]
      & = \frac{\ell^T \frac{M_{r_1}}{\lambda_B^{r_1}} \Big( P_2 P_4 \cdots
      P_{2m} \Big) \Big( Q_{2m} \cdots Q_{i_x+4} Q_{i_x+2} \Big) N_{i_x+1}
      v_x}{c_x^o \lambda_C^{r_{i_x+1}}} \\[.5em]
      & \xrightarrow{m \to \infty}
      \frac{\ell^T \frac{M_{r_1}}{ \lambda_B^{r_1}} Y_2 Z_{i_x+2}
      N_{{i_x+1}} v_x }{c_x^o}
      = 
      \ell_o^T \left(\frac{Z_{i_x+2}}{c_x^o}\frac{N_{{i_x+1}}}{
      \lambda_C^{i_x+1}} \right) v_x 
    \end{align*}
    Now suppose $i_x$ is odd. Then for all even $2m \ge i_x$, we have
    \begin{align*}
      \norm{x}_{ \left( \tau_{2m}, \frac{\ell_{2m}}{c_{2m}}\right)} 
      & = \frac{\ell^T \Big(P_1 P_3\cdots P_{2m-1} \Big) \Big( Q_{2m-1}
      \cdots Q_{i_x+3} \Big) N_{r_{i_x+1}} v_x}{c_x^e \lambda_C^{r_{i_x+1}}} \\[1em]
      & \xrightarrow{m \to \infty} \ell_e^T \left( \frac{
        Z_{i_x+2}}{c_x^e}\frac{N_{r_{i_x+1}}}{\lambda_C^{r_{i_x+1}}} \right) v_x.
    \end{align*}
    and for odd $2m+1 \ge i_x$, we have
    \begin{align*}
      \norm{x}_{ \left( \tau_{2m+1}, \frac{\ell_{2m+1}}{c_{2m+1}}\right)} 
      & = \frac{\ell^T \frac{M_{r_1}}{\lambda_B^{r_1}} \Big( P_2
      P_4 \cdots
      P_{2m} \Big) \Big( Q_{2m} \cdots Q_{i_x+3} Q_{i_x+1} \Big)
      v_x}{c_x^o} \\[.5em]
      & \xrightarrow{m \to \infty}
      \ell_o^T \left(\frac{Z_{i_x+1}}{c_x^o} \right) v_x.
    \end{align*}
    Either way, for any conjugacy class $x$ in $\free[7]$, both
    \[\norm{x}_{T_e} =
      \lim_{m \to \infty} \norm{x}_{ \left( \tau_{2m},
      \frac{\ell_{2m}}{c_{2m}}\right)}
      \qquad \text{and} \qquad
      \norm{x}_{T_o} =
      \lim_{m \to \infty}
      \norm{x}_{ \left( \tau_{2m+1}, \frac{\ell_{2m+1}}{c_{2m+1}}\right)} 
    \]
    are well-defined and have the desired form. \qedhere 
  \end{proof}
   
  We now want to show $T_e$ and $T_o$ are not scalar multiples of each
  other. In fact, the following lemma will allow us to show that $T_e$ and
  $T_o$ are the extreme points of the simplex $\P\D(T)$. 

  \begin{lemma}\label{lem:ratio}
    There exist two sequences $\alpha_i$ and $\beta_i$ of conjugacy classes
    of elements of $\free[7]$ such that the following holds. For any
    positive length vector $\ell=\ell_0$ on $\tau_0$, let $T_e$ and $T_o$
    be the respective limiting trees in $\partial \CV[7]$ for $\left(
    \tau_{2m}, \frac{\ell_{2m}}{c_{2m}} \right)$ and $\left( \tau_{2m+1},
    \frac{\ell_{2m+1}}{c_{2m+1}} \right)$. Then
    \[ \frac{||\alpha_i||_{T_o}}{||\alpha_i||_{T_e}} \xrightarrow{i \to
    \infty} \infty, \qquad \text{and} \qquad
    \frac{||\beta_i||_{T_o}}{||\beta_i||_{T_e}}
    \xrightarrow{i \to \infty} 0.\]
  \end{lemma}

  \begin{proof}
    Take the letter $e \in \free[7]$ and recall the automorphisms $\Phi_i$
    used to define the folding and unfolding sequences. Set $x_i =
    \Phi_i(e)$. For each $i$, $x_i$ is legal in $\tau_i'$ and is
    represented by the vector $e_5=(0,0,0,0,1,0,0)^T$ in $\tau_i'$.

    Using notation from \Cref{lem:TeTo}, set $c_{i}^e = c_{x_i}^e$ and
    $c_{i}^o = c_{x_i}^e$. Note here $i$ is the smallest index such that
    $x_i$ is legal in $\tau_i'$. We compare the ratio of $c_i^o$ and
    $c_i^e$. Since $r_{i+1} - r_i \to \infty$, we have
    \[
      \frac{c_{2i}^e}{c_{2i}^o} = 
      \frac{\lambda_C^{r_2} \cdots \lambda_C^{r_{2i}}}
           {\lambda_C^{r_1} \cdots \lambda_C^{r_{2i-1}}}
      \xrightarrow{i \to \infty} \infty,
      \qquad
      \text{while}
      \qquad
      \frac{c_{2i+1}^e}{c_{2i+1}^o} = 
      \frac{\kappa_C}{\lambda_C^{r_1}}
      \frac{\lambda_C^{r_2} \cdots \lambda_C^{r_{2i}}}
           {\lambda_C^{r_3} \cdots \lambda_C^{r_{2i+1}}}
      \xrightarrow{i \to \infty} 0. 
    \]
    
    Recall that both $\ell_e$ and $\ell_o$ are positive and by
    \Cref{lem:pairprod2} the sequence $Z_i $ converges to $Z$. Since $Ze_5$
    is the zero vector, by continuity of the
    dot product, 
    \[  \lim_{i \to \infty} \ell_e^T Z_{2i+1} e_5 = \ell_e^T Z e_5 = 0 
    \qquad \text {and} \qquad \lim_{i \to \infty} \ell_o^T Z_{2i+1} e_5 =
    \ell_o^T Z e_5 = 0. \]
    
    Next let $N_\infty = \lim_{i \to \infty} \frac{N_{r_i}}{
    \lambda_C^{r_i}}$ and recall by  \Cref{lem:MYMZ} that the vector $Z
    N_\infty e_5=(\star,\star,\star,0,0,0,0)$ is non-negative. Thus there
    are positive constants $A$ and $B$ such that, 
    \[  
    \lim_{i \to \infty} \ell_e ^T \left (Z_{2i+2}
    \frac{N_{r_{2i+1}}}{\lambda_C^{r_{2i+1}}}\right ) e_5 =
    \ell_e^T  ZN_\infty e_5
    = A > 0, \]
    and
    \[  
    \lim_{i \to \infty} \ell_o^T \left(Z_{2i+2} \frac{N_{r_{2i+1}}}{
    \lambda_C^{r_{2i+1}}}\right) e_5 = \ell_o^T  Z N_\infty e_5
    = B > 0.\]
    
    Combining the above observations and the formulas for length of $x_i$
    in $T_e$ and $T_o$ obtained in \Cref{lem:TeTo} we get:
    \[
    \frac{\norm{x_{2i}}_{T_o}}{\norm{x_{2i}}_{T_e}}
      =
      \frac{\ell_o^T \left (Z_{2i+2}
      \frac{N_{r_{2i+1}}}{\lambda_C^{r_{2i+1}}}\right ) e_5}{\ell_e^T \left
      ( Z_{2i+1} \right ) e_5}
      \frac{c_{2i}^e}{c_{2i}^o} \quad \xrightarrow{i \to \infty} \quad
      \frac{A}{0} \cdot \infty 
    \]
     \[
      \frac{\norm{x_{2i+1}}_{T_o}}{\norm{x_{2i+1}}_{T_e}}
      =
      \frac{\ell_o^T \left ( Z_{2i+1} \right ) e_5}{\ell_e^T \left
      (Z_{2i+2} \frac{N_{r_{2i+1}}}{ \lambda_C^{r_{2i+1}}}\right ) e_5}
      \frac{c_{2i+1}^e}{c_{2i+1}^o} \quad \xrightarrow{i \to \infty} \quad
      \frac{0}{B} \cdot 0 
    \]
    Setting $\alpha_i = x_{2i}$ and $\beta_i= x_{2i+1}$ finishes the proof.
    \qedhere 
  \end{proof}

  \begin{cor}\label{cor:ergotrees}

    For a sequence $(r_i)_{i \ge 1}$ with $r_{i+1}-r_i \ge i$, if the
    folding sequence $(\tau_i')_i$ converges to an arational tree $T$, then
    for any positive length vector $\ell_0$ on $\tau_0$, the limit set in
    $\partial \CV[7]$ of the rescaled unfolding sequence $(\tau_i,\ell_i)$
    is always the 1-simplex $\P\D(T)$. 

  \end{cor}

  \begin{proof}
    Since the folding $(\tau_i)'_i$ and the unfolding sequence $(\tau_i)_i$
    are equal as marked graphs for all $i \ge 0$, no matter the metric,
    they both visit the same sequence of simplices in $\CV[7]$. In particular,
    they both project to the same quasigeodesic in $\FreeF[7]$. Thus, the
    two limiting trees $T_e$ and $T_o$ of the even and odd sequences of
    $(\tau_i,\ell_i)$ are length measures on $T$.  
    
    Recall $\P\D(T)$ is a $1$-simplex by \Cref{cor:trees 1-simplex}. If
    neither $T_e$ nor $T_o$ are the extreme points of this simplex, then
    there exist constants $c, c'>0$ such that any $x \in \free$, 
    \[c'\leq \frac{\norm{x}_{T_o}}{\norm{x}_{T_e}} \leq c.\]  
    On the other hand, if one of them, say $T_o$, is an extreme point but
    $T_e$ is not, then we have a constant $c>0$ such that for any $x \in
    \free$, $ \frac{\norm{x}_{T_o}}{\norm{x}_{T_e}} \leq c.$ In both the
    cases, we get a contradiction to \Cref{lem:ratio}.  \qedhere 

  \end{proof}

\section{Conclusion}

  \label{sec:conclusion}

  Recall $\phi\in \Aut(\free[7])$ is the automorphism: 
  \[
    a\mapsto b,b\mapsto c,c\mapsto ca,d\mapsto d,e\mapsto e,f\mapsto
  f,g\mapsto g
  \]
  and $\rho\in \Aut(\free[7])$ is the rotation by 4 clicks:
  \[ a\mapsto e, b\mapsto f, c\mapsto g, d\mapsto a, e\mapsto b, f\mapsto
  c, g\mapsto d.\] For any integer $r$, let $\phi_r = \rho \phi^r$. To each
  sequence $(r_i)_{i \ge 0}$ of positive integers, we have an unfolding
  sequence $(\tau_i)_i$ with train track map $\phi_{r_i}\from \tau_i \to
  \tau_{i-1}$, and a folding sequence $(\tau_i')_i$ with train track map
  $\phi_{r_i}^{-1}\from \tau_{i-1}' \to \tau_i$. By the limit set of the
  unfolding sequence $(\tau_i)_i$ in $\partial \CV$ we mean the limit set
  of $(\tau_i,\ell_i)$ with respect to some (any) positive length vector
  $\ell_i$ on $\tau_i$.   

  \begin{main} \label{thm:conclusion}
  
    Given a strictly increasing sequence $(r_i)_{i \geq 1}$ satisfying $r_i
    \equiv i \mod 7$ and $r_i \equiv 0 \mod 3$, then the folding sequence
    $(\tau'_i)_i$ converges to a non-geometric arational tree $T$. 

    If $(r_i)_i$ grows fast enough, that is, if $r_{i+1}-r_i \ge i$, then
    $T$ is both non-uniquely ergometric and non-uniquely ergodic. Both
    $\P\D(T)$ and $\P\C(T)$ are 1-dimensional simplices.

    Furthermore, the limit set in $\partial \CV[7]$ of the unfolding
    sequence $(\tau_i)_i$ is always the 1-simplex spanned by the two
    ergodic metrics on $T$. 
  \end{main}

  \begin{proof}
    A sequence as in the statement exists by the Chinese remainder theorem.
    The first statement follows from \Cref{cor:arational} and
    \Cref{prop:nongeometric}. Non-unique ergometricity of $T$
    follows from \Cref{prop:IsoLength} and
    \Cref{cor:trees 1-simplex}. Non-unique ergodicity of $T$ is
    \Cref{cor:nue}. Finally, the last statement is
    \Cref{cor:ergotrees}. \qedhere 
  \end{proof}

\section{Appendix}

  \label{sec:appendix}

  \subsection{Convergence Lemma}

  Let $\norm{\cdot}$ denote the operator norm. Thus $\norm{Y}\geq 1$ for a
  nontrivial idempotent matrix $Y$.

  \begin{lemma}\label{lem:convergence}
    Let $Y$ be an idempotent matrix and $\Delta_i$, $i\geq 1$, a sequence of
    matrices with $\norm{\Delta_i}\leq \frac {\epsilon}{2^i}$ for some
    $\epsilon>0$. Assume also that $\epsilon \norm{Y}\leq 1/2$. Then the
    infinite product $$\prod_{i=1}^\infty (Y+\Delta_i)$$
    converges to a matrix $X$ with $\norm{X-Y}\leq 2\epsilon
    \Big(\norm{Y}+\norm{Y}^2 \Big)$. Moreover, the kernel of $Y$ is
    contained in the kernel of $X$.
  \end{lemma}

  \begin{proof}

    Write $$Y+\Sigma_k=\prod_{i=1}^k (Y+\Delta_i)$$
    Then $(Y+\Sigma_k)(Y+\Delta_{k+1})=Y+\Sigma_{k+1}$ and since $Y^2=Y$
    it follows that
    \begin{equation}\label{eqn:convergence}
    \Sigma_{k+1}=Y\Delta_{k+1}+\Sigma_k(Y+\Delta_{k+1})
    \end{equation}

    Multiplying on the right by $Y$ and using $Y^2=Y$ we get
    $$\Sigma_{k+1}Y=Y\Delta_{k+1}Y+\Sigma_kY+\Sigma_k\Delta_{k+1}Y$$
    and applying the norm:
    $$\norm{\Sigma_{k+1}Y}\leq \norm{\Sigma_kY}+\frac
    {\epsilon}{2^{k+1}}
    \norm{Y}^2+\norm{\Sigma_k}\frac{\epsilon}{2^{k+1}}\norm{Y}$$
    By adding these for $k=1,2,\cdots,m-1$ and using $\Sigma_1=\Delta_1$
    we have
    \begin{align*}
      \norm{\Sigma_mY}
      & \leq \norm{\Sigma_1Y}+\epsilon \norm{Y}^2 \left(\frac 14+\cdots +\frac
      1{2^m}\right)+\epsilon\norm{Y}\left(
      \frac{\norm{\Sigma_1}}{4}+\cdots+\frac{\norm{\Sigma_{m-1}}}{2^m}\right)
      \\
      & \leq \epsilon
      \Big( \norm{Y}+\norm{Y}^2 \Big) + \epsilon\norm{Y}\left(
      \frac{\norm{\Sigma_1}}{4}+\cdots+\frac{\norm{\Sigma_{m-1}}}{2^m} \right)
    \end{align*}

    So the norms of $\Sigma_mY$ are bounded by norms of $\Sigma_i$ with
    $i<m$. From \Cref{eqn:convergence} we also see that the norm of
    $\Sigma_{k+1}$ is bounded by the norms of $\Sigma_kY$. Putting this
    together we have
    \begin{align*}
      \norm{\Sigma_{k+1}}
      &\leq
      \norm{Y}\norm{\Delta_{k+1}}+\norm{\Sigma_kY}+\norm{\Sigma_k}\norm{\Delta_{k+1}}\\
      &\leq
      \frac{\epsilon}{2^{k+1}} \norm{Y} + \norm{\Sigma_kY} +
      \frac{\epsilon}{2^{k+1}} \norm{\Sigma_k}\\ 
      &\leq
      \frac{\epsilon}{2^{k+1}} \norm{Y} + \frac{\epsilon}{2}
      \Big( \norm{Y}+\norm{Y}^2 \Big) + \epsilon \norm{Y} \left(
      \frac{\norm{\Sigma_1}}{4} + \cdots +
      \frac{\norm{\Sigma_{k-1}}}{2^k} \right) + \frac{\epsilon}{2^{k+1}} \norm{\Sigma_k} \\
      &\leq
      \epsilon \Big(\norm{Y}+\norm{Y}^2 \Big) + \epsilon\norm{Y} \left(
      \frac{\norm{\Sigma_1}}{4}+\cdots+\frac{\norm{\Sigma_{k-1}}}{2^k}
      + \frac{\norm{\Sigma_k}}{2^{k+1}}\right). 
    \end{align*}
    Thus we have an inequality of the form
    \[ \norm{\Sigma_{k+1}}\leq
      a + b \left( \frac{\norm{\Sigma_1}}{4} +
      \cdots + \frac{\norm{\Sigma_k}}{2^{k+1}} \right)
    \]
    for $a = \epsilon \Big(\norm{Y} + \norm{Y}^2 \Big)$ and $b = \epsilon
    \norm{Y}$. 
    
    Set $c = 2 \epsilon \Big( \norm{Y}+\norm{Y}^2 \Big)$. Then $c \ge
    \epsilon$, $a \le c/2$, and $b \leq 1/2$ by assumption. Easy induction
    then shows for all $k \ge 1$, 
    \begin{equation}\label{eqn:sigma_k}
      \norm{\Sigma_k}\leq c.
    \end{equation} 
    This obtains the inequality $\norm{X-Y}\leq c$ from the statement, once
    we establish convergence.

    To see convergence we argue that the sequence of partial products
    forms a Cauchy sequence. For $1<k<m$,
    \begin{align*}
      \prod_{i=1}^m (Y+\Delta_i) - \prod_{i=1}^k (Y+\Delta_i)=
      \prod_{i=1}^{k-1} (Y+\Delta_i) \bigg(\prod_{i=k}^m (Y+\Delta_i) -
      (Y+\Delta_k)\bigg)
    \end{align*}
    By \Cref{eqn:sigma_k}, the norm of $\prod_{i=1}^{k-1}
    (Y+\Delta_i) = Y + \Sigma_{k-1}$ is bounded by $c+\norm{Y}$. We can
    apply the same estimate to the sequence starting with $Y+\Delta_k$ and
    with $\epsilon$ replaced with
    $\frac {\epsilon}{2^{k-1}}$ to see that 
    \[\norm{\prod_{i=k}^m (Y+\Delta_i) - Y} \leq \frac
    {2\epsilon(\norm{Y}+\norm{Y}^2) }{2^{k-1}}\leq \frac {c}{2^{k-1}}\]
    and so 
    \[
      \norm{\prod_{i=k}^m (Y+\Delta_i) - (Y+\Delta_k)}\leq
      \frac{c}{2^{k-1}}+ \frac{1}{2^k} 
    \]
    which proves the sequence is Cauchy. 

    For the second statement, set $X_k = \prod_{i=k}^\infty (Y+\Delta_i)$
    for $k \ge 1$. By the same estimate as above with $\epsilon$ replaced
    with $\frac{\epsilon}{2^{k-1}}$, we know that $X_k$ exists and 
    \[ 
      \norm{X_k - Y} \le \frac{2 \epsilon}{2^{k-1}} \Big( \norm{Y} +
      \norm{Y}^2 \Big) = \frac{c}{2^{k-1}}.
    \]
    By definition, $X = (Y+\Sigma_k)X_{k+1}$. Suppose $v$ is a unit vector
    with $Yv = 0$. Then
    \begin{align*}
      \norm{Xv} 
      & \le \norm{Y+\Sigma_k} \norm{X_{k+1} v} \\
      & = \norm{Y+\Sigma_k} \norm{X_{k+1} v - Yv} \\
      & \le \norm{Y+\Sigma_k} \norm{X_{k+1} - Y} \\
      & \le (\norm{Y}+c) \frac{c}{2^k}.
    \end{align*}
    Since this is true for all $k \ge 0$, letting $k \to \infty$ yields
    $X v = 0$. \qedhere
    \end{proof}

%section{References}

  \bibliography{nue}

  \end{document}